\documentclass[reqno,11pt]{amsart}
\usepackage[top=1in, bottom=1in, left=1.2in, right=1.2in]{geometry}             
\usepackage[pdftex,dvipsnames]{xcolor}
\usepackage{graphicx}
\usepackage{amssymb}
\usepackage{epstopdf}

\usepackage{tikz-cd}
\usepackage{enumitem}
\usepackage{hyperref}
\usepackage{aliascnt}
\usepackage{amsmath}
\usepackage{stmaryrd}
\usepackage[colorinlistoftodos,prependcaption,textsize=tiny, obeyFinal]{todonotes}
  \makeatletter
 \providecommand\@dotsep{5}
 \def\listtodoname{List of Todos}
 \def\listoftodos{\@starttoc{tdo}\listtodoname}
 \makeatother

\def\bC{{\mathbb{C}}}

\def\bP{{\mathbb{P}}}

\def\bQ{{\mathbb{Q}}}

\newcommand{\sP}{\mathsf{P}}
\newcommand{\sB}{\mathsf{B}}
\newcommand{\sH}{\mathsf{H}}
\newcommand{\aR}{\mathcal{R}}
\newcommand{\aD}{\mathcal{D}}

\newcommand{\N}{\mathbb{N}}
\newcommand{\Z}{\mathbb{Z}}

\newcommand{\R}{\mathbb{R}}

\renewcommand{\P}{\mathbb{P}}

\newcommand{\sm}{\setminus}

\newcommand{\F}{\mathcal F}
\newcommand{\LL}{\mathcal L}
\newcommand{\J}{\mathcal J}
\newcommand{\K}{\mathcal K}

\newcommand{\T}{\mathcal T}

\newcommand{\SH}{\mathrm{SH}}

\DeclareMathOperator{\PGL}{PGL}

\DeclareMathOperator{\Crit}{Crit}
\DeclareMathOperator{\Hull}{Hull}

\DeclareMathOperator{\diam}{diam}

\DeclareMathOperator{\Aut}{Aut}

\newcommand{\mynewtheorem}[3]{
\expandafter\def\csname#1autorefname\endcsname{#3}
\newaliascnt{#1}{thm}
\newtheorem{#1}[#1]{#2}
\aliascntresetthe{#1}
}

\theoremstyle{plain}
\newtheorem{thm}{Theorem}[section]
\newtheorem{prop}[thm]{Proposition}
\newtheorem{lem}[thm]{Lemma}
\newtheorem{cor}[thm]{Corollary}

\theoremstyle{definition}
\newtheorem{defn}[thm]{Definition}

\newtheorem{rmk}[thm]{Remark}

\makeatletter
\newcommand{\norm}[2][\@nil]{%
  \def\tmp{#1}%
   \ifx\tmp\@nnil
       \left\|#2\right\|
    \else
         \left\|#2\right\|_{#1}
    \fi}
\newcommand{\set}[2][\@nil]{%
  \def\tmp{#1}%
   \ifx\tmp\@nnil
       \left\{#2\right\}
    \else
         \left\{#1:\, #2\right\}
    \fi}
\makeatother

\numberwithin{equation}{section}
\title{Berkovich dynamics of twisted rational maps}

\author{Hongming Nie}
\address{Institute for Mathematical Sciences, Stony Brook University, Stony Brook NY 11794, USA}
\email{hongming.nie@stonybrook.edu}

\author{Shengyuan Zhao}
\address{Universit\'e Paul Sabatier, Institut de Mathématiques de Toulouse, 118, route de Narbonne, F-31062 Toulouse, France}
\email{shengyuan.zhao@math.univ-toulouse.fr}

\begin{document}

\begin{abstract}
A twisted rational map over a non-archimedean field $K$ is the composition of a rational map over $K$ and a continuous automorphism of $K$. We explore the dynamics of some twisted rational maps on the Berkovich projective line. 
   
\end{abstract}

\maketitle

\tableofcontents

\section{Introduction}\label{sec:intro} 

Let $K$ be an algebraically closed field of characteristic $0$, complete with respect to a non-trivial and non-archimedean absolute value $|\cdot|$. Let $\tau$ be a continuous field automorphism of $K$ such that
\begin{align}\label{eq:fieldautomorphismcondition} 
|\tau(\cdot)|=|\cdot|^{\lambda_\tau} 
\end{align}
for some positive real constant $\lambda_\tau$. In the case where $K$ is the completion of the field of Puiseux series over a field of characteristic $0$, the constant $\lambda_\tau$ can take any positive rational value and those $\tau$ with $\lambda_\tau=1$ form an infinite non-abelian group (see Deschamps \cite{Deschamps05} for detailed descriptions). In the case where $K=\bC_p$, any element in $\mathrm{Gal}(\overline{\bQ}_p/\bQ_p)$ satisfies \eqref{eq:fieldautomorphismcondition}  with  $\lambda_\tau=1$ (see \cite{KConrad} and MacLane \cite{MacLane39}).

Let $\phi(z)=\frac{\sum_{k=0}^m a_k z^k}{\sum_{k=0}^n b_k z^k}\in K(z)$ be a rational function. We consider the following map denoted by $\phi_\tau$:
\[
\phi_\tau: z\mapsto \frac{\sum_{k=0}^m a_k \tau(z)^k}{\sum_{k=0}^n b_k \tau(z)^k}.
\]
Such a map is the composition of the rational map $\phi$ and the field automorphism $\tau$; we call it a {\it $\tau$-twisted rational map} or simply a {\it twisted rational map}. A twisted rational map does not act on the projective line as an endomorphism of $K$-algebraic variety; in algebro-geometric language it fits the following commutative diagram of morphisms of schemes:
\[
\begin{tikzcd}
\bP^1_{K} \arrow{r}{\phi_\tau} \arrow{d}{} & \bP^1_{K}\arrow{d}{}\\
\operatorname{Spec}(K) \arrow{r}{\tau^*}& \operatorname{Spec}(K).
\end{tikzcd}
\]
A twisted rational map acts continuously with respect to the non-archimedean topology on the set of rational points $\bP^1(K)$. One can identify $\bP^1(K)$ with a dense subset of the Berkovich projective line $\mathsf{P}^1(K)$ and extend the action continuously to $\mathsf{P}^1(K)$. Such twisted actions on $\mathsf{P}^1(K)$ already appeared in a paper of Rumely (see \cite[Proposition 6.8]{Rumely17}). 

Dynamics of non-archimedean rational maps 
have been extensively studied since the work of 
 Benedetto \cite{Benedetto98} and  Rivera-Letelier \cite{Rivera03II} in analogy to classical dynamics of complex rational maps acting on $\P^1(\bC)$. The analogue of our twisted rational map in the complex setting is anti-holomorphic rational map of the form $z\mapsto \frac{\sum_{k=0}^m a_k \bar{z}^k}{\sum_{k=0}^n b_k \bar{z}^k}$ where $\bar{z}$ denotes the complex conjugation. As complex conjugation is a field automorphism of $\bC$ of order two, an anti-holomorphic rational map becomes a rational map after one iteration and asymptotic properties of anti-holomorphic rational maps reduce to that of rational maps. However a non-archimedean field $K$ could have natural field automorphism $\tau$ satisfying our condition \eqref{eq:fieldautomorphismcondition} and of infinite order. Therefore there are twisted rational maps whose iterates are never rational (see Remark \ref{rmk:not}). 

In this paper we explore the Fatou-Julia theory of twisted rational maps on $\mathsf{P}^1(K)$. Despite some fundamental differences between rational maps and twisted rational maps, the Fatou-Julia theories of these two classes of maps share many basic features. 
We will simply write $\bP^1$ and $\sP^1$ instead of $\bP^1(K)$
and $\sP^1(K)$, 
 unless specified otherwise.

We denote by $\F(\phi_\tau)$  and $\J(\phi_\tau)$ the Fatou set and the Julia set of $\phi_\tau$ in $\mathsf{P}^1$ and by $\F_I(\phi_\tau)$ and $\J_I(\phi_\tau)$ 
the Fatou set and the Julia set in $\bP^1$ (see Section \ref{prelim} for definitions). 

\begin{thm}\label{thm:F-J}
Let $\phi_\tau$ be a twisted rational map with $\deg\phi\ge 1$. Then 
$$\F_I(\phi_\tau)\subseteq \F(\phi_\tau)\cap\P^1\ \ \text{and}\ \ \J_I(\phi_\tau)\supseteq\J(\phi_\tau)\cap\P^1.$$
Moreover, if $\lambda_\tau\ge 1$, then 
$$\F_I(\phi_\tau)= \F(\phi_\tau)\cap\P^1\ \ \text{and}\ \ \J_I(\phi_\tau)=\J(\phi_\tau)\cap\P^1.$$
\end{thm}

Equidistribution of backward orbits for rational maps of degree at least $2$ has been established by Favre and Rivera-Letelier in \cite{Favre04, Favre10} 
(see also Baker and Rumely \cite[Section 10.3]{Baker10},  Chambert-Loir \cite{Chambert-Loir06}, Jonsson \cite[Section 5.7]{Jon15} and Thuillier \cite{Thuillier05}). We have the same phenomenon for 
 twisted rational maps: 
\begin{thm}\label{thm:equ}
Let $\phi_\tau$ be a twisted rational map of degree $d\ge 2$.  Assume that $d\cdot\lambda_\tau > 1$. Then there exists a unique Radon probability measure $\mu$ on $\mathsf{P}^1$ with the following property: if $\nu$ is a Radon probability measure on $\mathsf{P}^1$, then the weak convergence
\[\nu_n:=\frac{1}{d^n}(\phi_\tau^n)^\ast\nu\rightarrow \mu\quad \text{as}\ n\rightarrow \infty\]
holds if and only if $\nu(E(\phi_\tau))=0$ where $E(\phi_\tau)$ is the exceptional set of $\phi_\tau$. The measure $\mu$ does not charge any classical point and satisfies $\phi_\tau^\ast\mu=d\cdot\mu$. 
\end{thm}

A rational map of degree at least two has a non empty Fatou set (see Benedetto \cite{Benedetto01}), so does a twisted rational map when $\lambda_\tau\geq 1$ (see Proposition \ref{prop:repelling}). Periodic Fatou components of rational maps are classified into two types by Rivera-Letelier \cite{Rivera03II} and Kiwi \cite[Appendix]{DeMarco16}. 
The same dichotomy holds for twisted rational maps when $\lambda_\tau=1$:
\begin{thm}\label{thm:periodic}
Let $\phi_\tau$ be a tame twisted rational map of degree $d\ge 2$. Assume that $\lambda_\tau=1$. If $U\subset\F(\phi_\tau)$ is a fixed Fatou component, then either
\begin{enumerate}
\item $U$ is an attracting domain; or 
\item $U$ is a Rivera domain with $\partial U$ consisting of at most $d-1$ type II periodic orbits.
\end{enumerate}
Either of cases (1) and (2) may occur.
\end{thm}

The following description of wandering Fatou domains is also the same as in the case of rational maps in \cite{Benedetto00, Benedetto05}. 

\begin{thm}\label{thm:nonwandering}
Let $k\subset K$ be a discretely valued subfield with algebraically closure $\overline{k}\subset K$. Suppose that for any $\overline{k}$-rational point $\varpi$, there exists a discretely valued subfield $L$ with $k(\varpi)\subset L\subset K$ such that $\tau$ preserves $L$. Then for any tame twisted rational map $\phi_\tau$ of degree at least $2$, any wandering domain of $\F(\phi_\tau)$ containing a $\overline{k}$-rational point is contained in the basin of a type II Julia periodic cycle. 
\end{thm}

Trucco's description of  Julia sets of certain polynomials in \cite{Trucco14} also applies to twisted polynomials; here we describe the Julia dynamics of twisted polynomials with only escaping critical points: 
\begin{thm}\label{thm:poly}
Let $P_\tau$ be a twisted polynomial such that $P\in K[z]$ is a tame polynomial of degree $d\ge 2$ and satisfies $d\cdot\lambda_\tau>1$. If all critical points of $P_\tau$ are contained in the basin of $\infty$. Then on $\J(P_\tau)$, the map $P_\tau$ is topologically conjugate to the one-sided shift on $d$ symbols.
\end{thm}

Dynamics of twisted rational maps 
has potential applications in holomorphic dynamics of skew product rational maps on $\bC^2$ of the form 
\[
(x,y)\dashrightarrow (p(x),q(x,y))
\]
where $p,q$ are complex rational functions. Applying one dimensional twisted dynamics to two dimensional holomorphic dynamics is a main motivation of our work, which is an extension of recent fruitful achievements of non-archimedean/Berkovich dynamics tools in holomorphic dynamics (see \cite{Baker11, Serge18, DeMarco14, DeMarco16, Favre20, Favre07, Favre11, Ghioca12, Ji23, Kiwi15, Kiwi23, Luo22, Nie20, Nie22}).
 In a previous paper \cite{Zhao22}, the second author classified pairs of commuting birational transformations of $\bC^2$ by using degree one twisted rational maps. In his PhD thesis, Richard Birkett announced that one can build algebraically stable models for certain two dimensional skew product on $\bC^2$, with his independently developed similar tools of twisted rational maps. For other applications see our subsequent papers.

We end this introduction by mentioning that once the terminologies are settled, most proofs in our paper are similar to the case of rational maps and hence  we often omit the proof when it can be directly transported from the case of rational maps. 

\subsection*{Acknowledgement}
This project begins from a conversation between Junyi Xie and the second author. We thank Junyi for his insight and encouragement. We also thank Richard Birkett, who has independently obtained similar results with different terminologies in his PhD thesis (see \cite{Birkett23}), for his comments on our earlier drafts.

\section{Preliminaries}\label{prelim}

In this section, we provide some background materials on the Berkovich projective line and then introduce actions of twisted rational maps on the Berkovich projective line. After that, we will present some basic properties of twisted rational maps. 

\subsection{The Berkovich projective line}
Standard references are \cite{Baker10, Ben19, Berkovich90, Jon15}. As a set, the Berkovich projective line $\sP^1$ over $K$ consists of all (generalized) multiplicative seminorms on $K(z)$ which restrict to $|\cdot|$ on $K$. Points in $\sP^1$ are usually classified into four types as follows. For any type I, II or III point $\xi\in\sP^1\setminus\{\infty\}$, there exists a unique $K$-closed disk in $\bP^1$ of the form
\[
\overline{D}_\xi=\overline{D}(a,r):=\set[x\in K]{|x-a|\le r}\ \ \text{for}\ \ a\in K\ \ \text{and}\ \ r\ge 0
\]
 such that $\xi$ is the supremum seminorm $\left\|\cdot\right\|_\xi$ defined by $\left\|h\right\|_\xi=\sup_{x\in\overline{D}_\xi}|h(x)|$ for any $h\in K[z]$. Such a point $\xi$ is of type I if $r=0$, of type II if $r\in|K^\times|$, and of type III if $r\not\in|K^\times|\cup\{0\}$.  A type IV point in $\sP^1$ corresponds to (a cofinal equivalence class of) a nested decreasing sequence  of $K$-closed disks with empty intersection. The point $\infty$ corresponds to the evaluation $h\to |h(\infty)|$ for any $h\in K(z)$ and is a type I point. Thus the projective line $\bP^1$ over $K$ is canonically embedded in $\sP^1$ as the set of all type I points. The type II point corresponding to the unit closed disk in $\bP^1$ is called the \emph{Gauss point}, and is denoted by $\xi_G$. If $\xi\in\sP^1\setminus\{\infty\}$ is a point of type I, II or III, then its \emph{diameter} $\diam(\xi)$ is defined to be the diameter of the closed disk $\overline{D}_\xi$ corresponding to $\xi$; in this case we write $\xi=\xi_{a,r}$ where $a$ is any point in $\overline{D}_\xi$ and $r=\diam(\xi)$. The \emph{diameter} of a type IV point is the decreasing limit of diameters of the corresponding sequence of disks.

There is a natural partial ordering 
on $\sP^1$ induced by the inclusion relation among all $K$-closed disks, which gives a tree structure on $\sP^1$. The topology on $\sP^1$ is the weak topology. 
The space $\sP^1$ is Hausdorff, compact, connected, and contains $\bP^1$ as a dense subset.

If $\zeta, \xi$ are points in $\sP^1$, then we denote by $[\zeta, \xi]$ the \emph{segment} (with respect to the tree structure) in $\sP^1$ joining $\zeta$ and $\xi$, and define $]\zeta, \xi] = [\zeta, \xi]\setminus\{\zeta\}$, $[\zeta, \xi[= [\zeta, \xi]\setminus\{\xi\}$ and $]\zeta, \xi[= [\zeta, \xi]\setminus\{\zeta, \xi\}$. 
There is a unique point denoted by $\zeta \vee \xi$ in the intersection of the three closed segments $[\zeta, \xi]$, $[\zeta, \infty]$ and $[\xi,\infty]$; we have $[\zeta, \xi] = [\zeta, \zeta \vee \xi] \cup [\xi, \zeta \vee \xi]$.

At a point $\xi\in\sP^1$, connected components of $\sP^1\setminus\{\xi\}$ induce natural equivalence classes of points in $\sP^1\setminus\{\xi\}$. Each of these equivalence classes is called a \emph{tangent direction} at $\xi$. All tangent directions at $\xi$ form the \emph{tangent space} $T_\xi\sP^1$ of $\sP^1$ at $\xi$. If $\xi$ is of type I or IV, then $\#T_\xi\sP^1=1$; if $\xi$ is of type III, then $\#T_\xi\sP^1=2$; and if $\xi$ is of type II, then $T_\xi\sP^1$ can be identified with the projective line over the residue field of $K$. For any $\vec v\in T_\xi\sP^1$, denote by $\sB(\vec{v})$ its corresponding connected component of $\sP^1\setminus\{\xi\}$. A \emph{Berkovich open disk} is defined to be such a $\sB(\vec{v})$ for some $\xi$ and $\vec v\in T_\xi\sP^1$; the \emph{diameter} of $\sB(\vec{v})$ is $\mathrm{diam}(\xi)$.  Given $\zeta\in\sP^1\setminus\{\xi\}$, denote by $\vec{v}_\xi(\zeta)$ the unique direction at $\xi$ whose corresponding component contains $\zeta$. We sometimes denote by $\aD(0,1)$ the unit open Berkovich disk $\sB(\vec{v}_{\xi_G}(0))$ and by $\aD(\xi,r)$ the open Berkovich disk containing $\xi\in\sP^1$ with diameter $r>0$. 

The \emph{Berkovich hyperbolic space} is $\sH^1:= \sP^1\sm \bP^1$ equipped with a metric $\rho$ defined by
\begin{align}\label{eq:hyp-metric}
\rho(\zeta, \xi) = 2\log(\diam(\zeta \vee \xi)) - \log(\diam(\zeta)) - \log(\diam(\xi)),  \quad \zeta,\xi\in\sH^1.
\end{align}
The topology on $\sH^1$ induced by $\rho$ is finer than the relative topology induced by the weak topology on $\sP^1$. The set of type II points is dense in both $(\sH,\rho)$ and $\sP^1$. 

\subsection{Twisted rational maps}

\subsubsection{Definition}

Denote by $\Aut(K)$ the group of continuous field automorphisms of $K$ and by $\Aut^*(K)\subset\Aut(K)$ the subgroup consisting of $\tau\in\Aut(K)$ for which there is a real number $ \lambda_\tau \in\mathbb{R}_{>0}$ depending on $\tau$ such that
\[|\tau(\cdot)|= |\cdot|^{\lambda_\tau}.\label{eq:lambdatau}\]

Any element of $\Aut^*(K)$ acts as a homeomorphism on $\bP^1$ by fixing $\infty$. Observing that $\tau\in \Aut^*(K)$ maps a closed disk in $K$ to a closed disk in $K$, we can naturally extend the action of $\tau$ to $\sP^1$.
\begin{lem}\label{lem:fund1}
Let $\tau \in \Aut^*(K)$ with $\lambda=\lambda_\tau$. Then $\tau$ uniquely continuously extends to a homeomorphism on $\sP^1$ such that 
$\tau(\xi_{a, r}) = \xi_{\tau(a), r^\lambda}$ for any type I, II, or III point $\xi_{a,r}\in\sP^1$. In particular the type of a point in $\sP^1$ is preserved by $\tau$.
\end{lem}

\begin{proof}
We extend the action to $\sP^1$ by sending $\xi_{a, r}$ to $\xi_{\tau(a), r^\lambda}$ for any $a\in K$ and $r\ge 0$. In this way the preimage of any Berkovich disk under $\tau$ is a Berkovich disk. Therefore the action is continuous on $\sP^1$. The uniqueness of the extension follows from the density of type I points. To prove that the type of a point is preserved, it suffices to notice that if $r=\vert x \vert$ for some $x\in K^*$ then $r^{\lambda}=\vert \tau(x)\vert$ is also an absolute value.
\end{proof}

Pick $\tau\in\Aut^*(K)$ and let $\phi\in K(z)$ be a rational map. Then both $\tau$ and $\phi$ acts continuously on $\sP^1$ (see Lemma \ref{lem:fund1} and \cite[Section 7]{Ben19}). It follows that the composition $\phi\circ\tau (z)=\phi(\tau(z))$ is a continuous self-map of $\sP^1$. For any type I, II or III point $\xi_{a,r}\in\sP^1$, we have
\[
\phi\circ\tau(\xi_{a,r})=\phi(\xi_{\tau(a), r^{\lambda_\tau}}).
\]


\begin{defn}
If $\tau\in\Aut^*(K)$, then a \emph{$\tau$-twisted rational map} is a composition $\phi\circ \tau$, denoted by $\phi_\tau$, for some rational map $\phi\in K(z)$. We call $\phi_\tau:=\phi\circ\tau$ the \emph{$\tau$-twist} of $\phi$ and $\phi$ the associated rational map of $\phi_\tau$. 
\end{defn}

When the context is clear, we simply call $\phi_\tau$ a twisted rational map and write directly $\lambda_\tau$ for the 
factor in Section \ref{eq:lambdatau}.

\subsubsection{Basic properties}

\begin{lem}\label{lem:fund}
Let $\tau \in \Aut^*(K)$ with $\lambda=\lambda_\tau$. Then we have
\begin{enumerate}
\item for any $\zeta,\xi\in\sH^1$, $\tau([\zeta,\xi])=[\tau(\zeta),\tau(\xi)]$ and $\rho(\tau(\zeta),\tau(\xi))=\lambda\cdot\rho(\zeta,\xi)$; 
\item $\tau(\xi_G)=\xi_G$ and $\tau([0,\infty])=[0,\infty]$.
\end{enumerate}
\end{lem}

\begin{proof}
Let $\zeta,\xi\in\sH^1$. As $[\zeta,\xi]=[\zeta,\zeta\vee\xi]\sqcup(\zeta\vee\xi,\xi]$, it suffices to consider the case where 
$\xi\in[\zeta,\infty]$. Then 
$\tau(\xi)\in[\tau(\zeta),\infty]$. Applying Lemma \ref{lem:fund1} and \eqref{eq:hyp-metric}, we conclude that $\tau([\zeta,\xi])=[\tau(\zeta),\tau(\xi)]$ and 
\[
\rho(\tau(\zeta),\tau(\xi))=\log\frac{\diam(\tau(\xi))}{\diam(\tau(\zeta))}=\log\frac{\diam(\xi)^\lambda}{\diam(\zeta)^\lambda}=\lambda\cdot\rho(\zeta,\xi).
\]

By Lemma \ref{lem:fund1}, we have $\tau(\xi_{0, r}) = \xi_{\tau(0), r^\lambda}= \xi_{0, r^\lambda}$ because $\tau(0)=0$, which implies that $\tau([0,\infty])=[0,\infty]$. Taking $r=1$, we have $\tau(\xi_G)=\xi_G$. 
\end{proof}



For $\tau\in \Aut^*(K)$ and $\phi(z)=\frac{\sum_{k=0}^m a_k z^k}{\sum_{k=0}^n b_k z^k}\in K(z)$, we define $\widehat{\tau}(\phi)\in K(z)$ by
\[
\widehat{\tau}(\phi)(z):=\frac{\sum_{k=0}^m \tau(a_k) z^k}{\sum_{k=0}^n\tau(b_k) z^k}.
\]
A direct computation shows the following formula:

\begin{lem}\label{cor:composition}
Let $\tau\in \Aut^*(K)$ and let $\phi\in K(z)$ be a rational map. Then 
$$\phi_\tau=\tau\circ\widehat{\tau^{-1}}(\phi).$$
Moreover, for any $\eta,\beta \in \PGL(2, K)$, 
$$\eta\circ\phi_\tau\circ\beta=\tau\circ\widehat{\tau^{-1}}(\eta)\circ\widehat{\tau^{-1}}(\phi)\circ\beta.$$
\end{lem}

Since both $\tau$ and $\phi$ are open maps, we have 

\begin{lem}\label{cor:open}
A twisted rational map $\phi_\tau: \sP^1\to\sP^1$ is an open map if $\phi$ is a non-constant rational map.
\end{lem}


The following formula concerns the image of an annulus.

\begin{lem}\label{thm:imgtypeII}
 Let $\tau \in \Aut^*(K)$ and let $\phi\in K(z)$ be a non-constant rational map. Let $a\in K$ and pick $0<\theta<1$ sufficiently close to $1$ so that
 \[\phi_\tau(z) = \sum_{n \in \Z} b_n(\tau(z) - a)^n\] converges on the annulus
 \[U_\theta = \set{z\in K : \theta r < |z-\tau^{-1}(a)| < r}\] and such that $\phi(z) - b_0$ has both the inner and outer Weierstrass degrees equal to $d$.
 Then 
 \[\phi_\tau(U_\theta) = 
\begin{cases}
 \left\{z\in K : b_d(\theta r)^{\lambda d} < |z-b_0| < b_dr^{\lambda d}\right\} &\ \ \text{if}\ \ d > 0\\
 \left\{z\in K : b_dr^{\lambda d} < |z-b_0| < b_d(\theta r)^{\lambda d}\right\} &\ \ \text{if}\ \ d < 0.
\end{cases}
\]
Moreover, $\phi_\tau(\xi_{\tau^{-1}(a),r}) = \xi_{b_0, b_dr^{\lambda d}}.$
\end{lem}
\begin{proof}
It follows immediately from \cite[Theorem 7.12]{Ben19} and Lemma \ref{lem:fund1}.
\end{proof}


\subsection{Tangent map}
Let $\tau \in \Aut^*(K)$ and let $\phi\in K(z)$ be a non-constant rational map. Pick a point $\xi\in\sP^1$ and denote by $T_\xi\phi$ the tangent map of $\phi$ at $\xi$. By Lemma \ref{lem:fund1}, for any $\vec{v}\in T_\xi\sP^1$ and any $\zeta\in\sB(\vec{v})$, there is a unique $\vec{w}\in T_{\tau(\xi)}\sP^1$ such that $\tau(\zeta)\in\sB(\vec{w})$. This induces a map 
\[T_\xi\tau:T_\xi\sP^1\to T_{\tau(\xi)}\sP^1,\]
sending $\vec v$ to $\vec w$. Then we define the \emph{tangent map} of the twisted rational map $\phi_\tau$ to be
\[T_\xi\phi_\tau:=T_{\tau(\xi)}\phi\circ T_\xi\tau.\]

It follows from the definitions that tangents maps satisfy the chain rule:

\begin{lem}\label{lem:tangent}
Let $\tau,\upsilon\in \Aut^*(K)$ and let $\phi,\psi\in K(z)$ be non-constant rational maps.  Then for any $\xi\in\sP^1$,
$$T_\xi(\phi_\tau\circ\psi_{\upsilon})=T_{\psi_{\upsilon}(\xi)}\phi_\tau\circ T_\xi\psi_{\upsilon}.$$
\end{lem}

\begin{rmk}\label{rmk:not}
Applying Lemma \ref{lem:tangent}, we can show that there exist a field $K$ and $\tau\in \Aut^*(K)$ such that for any non-constant rational map $\phi\in K(z)$, the $n$-th iterate $\phi_\tau^n$ of the twisted rational map $\phi_\tau$ is not in $K(z)$ for any integer $n\ge 1$. A key property that any rational map $\psi\in K(z)$ satisfies is the following: for any type II point $\xi\in\sP^1$, and any $\vec{v}\in T_\xi\sP^1$, there exists $\xi_1\in\sB(\vec{v})$ such that the tangent map $T_\zeta\psi$ is independent of the point $\zeta$ in the segment $]\xi_1, \xi[$ (if we identify locally the tangent spaces at different points). This can be seen from the convergent series of $\psi$. Now we claim that this property does not hold for general twisted rational maps. Consider a field $K$ possessing a $\tau\in\Aut^*(K)$ such that $|\tau(x)|=|x|$ and $\tau(x)/x$ is not a root of unity in $K$ for any $x\in K^*$. Then for a non-constant rational map $\phi$, for any type II point $\xi\in]0,\xi_G[$ and for any $n\ge 1$, the tangent map $T_\zeta\phi_\tau^n$ of the $n$-th iterate $\phi_\tau^n$ is not constant on $]\xi, \xi_G[$. For example, the above phenomenon happens when $K=\overline{\mathbb{C}\{\{t\}\}}$ is the completion of the Puiseux series over $\mathbb{C}$ and $\tau\in\Aut^*(K)$ is an automorphism of $K$ sending $t$ to $2t$. We refer the reader to \cite{Deschamps05} for such automorphisms.
\end{rmk}

Tangents maps determine the images of open Berkovich disks in the following sense.
\begin{prop}\label{prop:image-whole}
Let $\tau \in \Aut^*(K)$ and let $\phi\in K(z)$ be a non-constant rational map. For any $\xi\in\sP^1$ and any direction $\vec{v}\in T_\xi\sP^1$, the image $\phi_\tau(\sB(\vec{v}))$ is either the whole space $\sP^1$ or the open Berkovich disk with boundary $\phi_\tau(\xi)$ corresponding to $T_\xi\phi_\tau(\vec{v})$. 
\end{prop}

\begin{proof}
The conclusion follows immediately from the fact that a rational map maps an open Berkovich disk to an open Berkovich disk or to the whole space (see \cite[Proposition 9.41]{Baker10}).
\end{proof}

We say that a direction $\vec{v}\in T_\xi\sP^1$ is \emph{good} for a twisted rational map $\phi_\tau$ if $\phi_\tau(\sB(\vec{v}))$ is an open Berkovich disk, and \emph{bad} if $\phi_\tau(\sB(\vec{v}))$ is $\sP^1$.

\subsection{Local degree and ramification}\label{sec:degree}
If $\phi\in K(z)$ is a rational map, then for any point $\zeta\in\sP^1$ and any direction $\vec{v}\in T_\zeta\sP^1$, one can define the local degree $\deg_\zeta\phi$ at $\zeta$, the directional multiplicity $m_{\vec{v}}(\phi)$ of $\vec{v}$ and surplus multiplicity $s_{\vec{v}}(\phi)$ of $\vec{v}$, see \cite{Faber13I},\cite[Chapter 7]{Ben19} and \cite[Chapter 9]{Baker10}. We can define these quantities for twisted rational maps as well.

For any $\tau \in \Aut^*(K)$, we set $\deg\tau:=1$, as suggested by Lemma \ref{lem:fund1}. Then the \emph{local degree} of $\tau$ at any point $\xi\in\sP^1$ is defined to be $1$, i.e.\ $\deg_\xi\tau:=1$. 

For any rational map $\phi\in K(z)$, we define the \emph{degree} of $\phi_\tau$ to be $\deg\phi_\tau:=\deg\phi$ and define the \emph{local degree}, the \emph{directional multiplicity} and the \emph{surplus multiplicity} of $\phi_\tau$ at any $\xi\in\sP^1$ and $\vec{v}\in T_\xi\sP^1$ to be those of $\phi$ at $\tau(\xi)$ and $T_\xi \tau (\vec{v})$, i.e.\ $\deg_\xi\phi_\tau:=\deg_{\tau(\xi)}\phi$, $m_{\vec{v}}(\phi_\tau):=m_{T_\xi \tau (\vec{v})}(\phi)$  and $s_{\vec{v}}(\phi_\tau):=s_{T_\xi \tau (\vec{v})}(\phi)$.

Now we state some properties (Propositions \ref{prop:deg-basic}, \ref{prop:multi-basic}, \ref{prop:deg-composition} and \ref{prop:deg-image}) for local degrees and multiplicities of twisted rational maps; all of them are easily obtained from the rational case, so we omit the proofs. 
\begin{prop}\label{prop:deg-basic}
Let $\phi_\tau$ be a twisted rational map with degree $d\ge 1$. For any $\xi\in\sP^1$, the following hold: 
\begin{enumerate}
\item $1\le\deg_\xi\phi_\tau\le d$.
\item $\sum_{\zeta\in\phi_\tau^{-1}(\xi)} \deg_\zeta\phi_\tau =d$.
\end{enumerate}
\end{prop}

\begin{prop}\label{prop:multi-basic}
Let $\phi$ be a twisted rational map with degree $d\ge 1$. For any $\xi\in\sP^1$ and $\vec{v}\in T_\xi\sP^1$, the following hold: 
\begin{enumerate}
\item There exists a point $\zeta \in \sB(\vec v)$ such that $\phi_\tau$ maps the segment $[\zeta , \xi]$ homeomorphically onto $[\phi_\tau(\zeta), \phi_\tau(\xi)]$, and
 \[\rho(\phi_\tau(\xi_1) , \phi_\tau(\xi_2)) = \lambda_\tau m_{\vec v}(\phi_\tau)\cdot \rho(\xi_1 , \xi_2),\quad \forall \xi_1, \xi_2 \in [\zeta , \xi]\cap\sH^1.\]
 In particular, $\phi_\tau$ is an isometry on $(\sH^1,\rho)$ if and only if $d = 1$ and $\lambda_\tau=1$.
\item The direction $\vec{v}$ is bad if and only if $s_{\vec{v}}(\phi_\tau)\ge 1$.
 \item $\deg_\zeta\phi_\tau+\sum_{\vec{v}\in T_\zeta\sP^1}s_{\vec{v}}(\phi_\tau)=d$.
\end{enumerate}
\end{prop}

The local degrees and multiplicities satisfy the following chain rules.
\begin{prop}\label{prop:deg-composition}
Let $\phi_\tau$ and $\psi_{\upsilon}$ be two twisted rational maps of degree at least $1$. For any $\xi \in \sP^1$, and any $\vec v \in T_\xi\sP^1$, the following hold:
\begin{enumerate}
\item $\deg_\xi(\psi_{\upsilon}\circ\phi_\tau) = \deg_{\phi_\tau(\xi)}\psi_{\upsilon}\cdot\deg_\xi\phi_\tau$.
\item $m_{\vec v}(\psi_{\upsilon}\circ\phi_\tau) =m_{T_\xi\phi_\tau(\vec v)}(\psi_{\upsilon})\cdot  m_{\vec v}(\phi_\tau)$.
\item $s_{\vec v}(\psi_{\upsilon}\circ\phi_\tau) =s_{T_\xi\phi_\tau(\vec v)}(\psi_{\upsilon})+ m_{\vec v}(\phi_\tau) \cdot s_{\vec v}(\phi_\tau)$.
\end{enumerate}
\end{prop}

These multiplicities count the number of preimages in the following sense.

\begin{prop}\label{prop:deg-image}

Let $\phi_\tau$ be a skew product of degree at least $1$ and pick $\xi\in\sP^1$. Then for any $\zeta \in \sP^1$, and any  $\vec v \in T_\zeta\sP^1$, if $\xi\in T_\zeta\phi_\tau(\vec{v})$, then $\sB(\vec{v})$ contains $m_{\vec v}(\phi_\tau)+s_{\vec v}(\phi_\tau)$ preimages in the set $\phi_\tau^{-1}(\{\xi\})$; and if  $\xi\not\in T_\zeta\phi_\tau(\vec{v})$, then $\sB(\vec{v})$ contains $s_{\vec v}(\phi_\tau)$ preimages in $\phi_\tau^{-1}(\{\xi\})$.
\end{prop}

The ramification locus $\aR(\phi)$ of a nonconstant rational map $\phi\in K(z)$ is defined in  \cite{Trucco14, Faber13I}, see also \cite[Section 7.6]{Ben19}. They can be defined in the same way for twisted rational maps. 

\begin{defn}
 Let $\phi_\tau$ be a twisted rational map with $\deg\phi_\tau\ge 1$. The \emph{ramification locus} of $\phi_\tau$ is the set
 $$\aR(\phi_\tau):= \{\zeta \in \sP^1: \deg_\zeta\phi_\tau\ge 2\}.$$
 \end{defn}
A \emph{(classical) critical point} of $\phi_\tau$ is a point $x\in\bP^1\subset\sP^1$ at which $\phi_\tau$ is not locally injective. Denote by $\Crit(\phi_\tau)$ the set of classical critical points of $\phi_\tau$. Then  $\Crit(\phi_\tau)\subset\aR(\phi_\tau)$.

\begin{lem}\label{lem:ram}
Let $\phi_\tau$ be  any twisted rational map of degree at least $1$. Then $\aR(\phi_\tau) = \tau^{-1}(\aR(\phi))$ and $\Crit(\phi_\tau) = \aR(\phi_\tau) \cap \bP^1=\tau^{-1}(\Crit(\phi))$.
\end{lem}

The ramification locus $\aR(\phi)$ has been well studied by Faber in \cite{Faber13I, Faber13II}. Due to Lemma \ref{lem:ram} (1), all the topological results on $\aR(\phi)$ hold for $\aR(\phi_\tau)$. In general, $\aR(\phi_\tau)$ may not be contained in the convex hull $\Hull(\Crit(\phi_\tau))$ of $\Crit(\phi_\tau)$. As in the rational case, following \cite{Trucco14}, we define the notion of tameness:
\begin{defn}\label{def:tame}
 Let $\phi_\tau$ be a twisted rational map with $\deg\phi_\tau\ge 1$. We say that $\phi_\tau$ is \emph{tame} if  $\aR(\phi_\tau)\subset\Hull(\Crit(\phi_\tau))$.
\end{defn}
Let $p\ge 0$ be the residue characteristic of $K$. If $p=0$ or $\deg\phi_\tau<p$, then $\phi_\tau$ is tame, see \cite[Corollary 6.6]{Faber13I}.
 
We end the section with a corollary of Proposition \ref{prop:multi-basic} (1). 

\begin{prop}\label{prop:fixedinterval}
Let $\phi_\tau$ be a twisted rational map of degree at least $1$. Pick two points $\zeta, \xi \in \sP^1$. Assume that $\phi_\tau$ is injective and has constant local degree $d\ge 1$ on $]\zeta,\xi[$. 
Then for any $\xi_1,\xi_2\in[\zeta,\xi]\cap\sH^1$,
 \begin{align}\label{eq:dis}
 \rho(\phi(\xi_1) , \phi(\xi_2)) = \lambda_\tau d\cdot \rho(\xi_1 , \xi_2).
 \end{align}
In particular, if $[\zeta, \xi]\cap \aR(\phi_\tau)=\emptyset$, then
 \[\phi_\tau : [\zeta, \xi] \to  [\phi(\zeta), \phi(\xi)]\] is a linear homeomorphism and  for any $\xi_1,\xi_2\in[\zeta,\xi]\cap\sH^1$,
 \[\rho(\phi(\xi_1) , \phi(\xi_2)) = \lambda_\tau\cdot \rho(\xi_1 , \xi_2)\]
\end{prop}
\begin{proof}
To show \eqref{eq:dis}, note that $\phi_\tau$ is injective on $]\zeta,\xi[$, and hence it maps $[\zeta,\xi]$ bijectively to $[\phi_\tau(\zeta),\phi_\tau(\xi)]$. Also observe that for any $\zeta_1\in[\zeta,\xi[$, the map $\phi_\tau$ has directional multiplicity $d$ at the direction $\vec{v}_{\zeta_1}(\xi)\in T_{\zeta_1}\sP^1$. 
Then \eqref{eq:dis} follows from Proposition \ref{prop:multi-basic} (1). 

Now consider the case where $[\zeta, \xi]\cap \aR(\phi_\tau)=\emptyset$. To see  $\phi_\tau : [\zeta, \xi] \to  [\phi_\tau(\zeta), \phi_\tau(\xi)]$ is a homeomorphism, it suffices to show that $\phi_\tau$ is injective on $[\zeta, \xi]$. Suppose on the contrary that $\phi_\tau$ is not injective on $[\zeta, \xi]$. Then there exist two distinct points in $[\zeta, \xi]$ having the same image under $\phi_\tau$. This implies the existence of $\xi'\in[\zeta, \xi]$ such that $\phi_\tau$ is not injective on any segment $I\subset ]\zeta, \xi[$ containing $\xi'$. Hence $\xi'\in\aR(\phi_\tau)$, which contradicts that $[\zeta, \xi]\cap \aR(\phi_\tau)=\emptyset$. The remaining assertions follow immediately. 
\end{proof}

\section{Julia and Fatou sets}

Let $\phi_\tau$ be a twisted rational map with $\deg\phi_\tau\ge 1$. In this section we introduce the Fatou/Julia set of $\phi_\tau$ and classify its periodic points.

\subsection{Fatou and Julia sets} 

For Fatou and Julia sets in $\sP^1$, we take the standard definition in Berkovich dynamics:
\begin{defn}
The \emph{(Berkovich) Fatou set} of $\phi_\tau$, denoted by $\F(\phi_\tau)$, is the subset of $\sP^1$ consisting of all points $\zeta \in \sP^1$ having a neighborhood  $U \subseteq \sP^1$  such that $\cup_{n \ge 0} \phi_\tau^n(U)$ omits infinitely many points of $\sP^1$. The \emph{(Berkovich) Julia set} of $\phi_\tau$, denoted by $\J(\phi_\tau)$, is the complement $\sP^1 \sm \F(\phi_\tau)$.
\end{defn}

For Fatou and Julia sets in $\bP^1$, we use the classical definition with equicontinuity. Denote by $\sigma$ the spherical distance on $\bP^1$. 

\begin{defn}
The \emph{(classical) Fatou set} of $\phi_\tau$, denoted by $\F_I(\phi_\tau)$, is the subset of $\bP^1$ consisting of all points $x \in \bP^1$ having a neighborhood on which the family of iterates $\{\phi_\tau^n\}_{n\ge 0}$ is equicontinuous with respect to $\sigma$. The \emph{(classical) Julia set} of $\phi_\tau$, denoted by $\J_I(\phi_\tau)$, is the complement $\bP^1 \sm \F_I(\phi_\tau)$.
\end{defn}

We first observe that $\phi_\tau$ is H\"older continuous on $(\bP^1,\sigma)$:
\begin{lem}\label{lem:holder}
Let $\phi_\tau$ be a twisted rational map with degree at least $1$. Then there exists $C\ge1$ such that for any $x,y\in\bP^1$, 
$$\sigma(\phi_\tau(x),\phi_\tau(y))\le C\sigma(x,y)^{\lambda_\tau}.$$
Moreover, if $\lambda_\tau\ge 1$, we have 
$$\sigma(\phi_\tau(x),\phi_\tau(y))\le C\sigma(x,y).$$
\end{lem}
\begin{proof}
 It follows from the Lipschitz continuity of rational maps \cite[Proposition 5.2]{Morton95} that there exists $C\ge 1$ such that 
 $$\sigma(\phi_\tau(x),\phi_\tau(y))\le C\sigma(\tau(x),\tau(y)).$$
 Now we consider $\sigma(\tau(x),\tau(y))$. Changing coordinates if necessary, we can assume that $\tau(x), \tau(y)\in K$. It follows that 
 $$\sigma(\tau(x),\tau(y))=\frac{|\tau(x)-\tau(y)|}{\max\{1,|\tau(x)|\}\max\{1,|\tau(y)|\}}=\frac{|x-y|^{\lambda_\tau}}{\max\{1,|x|^{\lambda_\tau}\}\max\{1,|y|^{\lambda_\tau}\}}=\sigma(x,y)^{\lambda_\tau}.$$

If $\lambda\ge 1$ then the assertion follows immediately since $\sigma(x,y)\le 1$
\end{proof}

Applying the same arguments as for rational maps (\cite[Proposition 8.2]{Ben19}), we immediately obtain the following statements about $\F(\phi_\tau)$ and $\J(\phi_\tau)$:
\begin{prop}\label{prop:fatoujuliabasics}
 Let $\phi_\tau$ be a twisted rational map with $\deg\phi_\tau\ge 1$. Then the following hold:
\begin{enumerate}
 \item $\F(\phi_\tau)$ is open and $\J(\phi_\tau)$ is closed.
 \item $\phi_\tau^{-1}(\F(\phi_\tau))=\F(\phi_\tau)= \phi_\tau(\F(\phi_\tau))$ and  $\phi_\tau^{-1}(\J(\phi_\tau))=\J(\phi_\tau)= \phi_\tau(\J(\phi_\tau))$.
 \item For every integer $m \ge 1$, $\F(\phi_\tau^m)= \F(\phi_\tau)$ and $\J(\phi_\tau^m)= \J(\phi_\tau)$.
 \item For any $\eta \in \PGL(2, K)$, if we set $\psi_\tau= \eta \circ\phi_\tau\circ \eta^{-1}$, then
 \[\F(\psi_\tau)= \eta(\F(\phi_\tau)) \quad \text{and} \quad \J(\psi_\tau)=\eta(\J(\phi_\tau)).\]
\end{enumerate}
\end{prop}

Now we show that the above properties also hold for 
$\F_I(\phi_\tau)$ and $\J_I(\phi_\tau)$, which is an analogue of \cite[Proposition 5.10]{Ben19}.
\begin{prop}\label{prop:fatoujuliabasics1}
 Let $\phi_\tau$ be a twisted rational map with $\deg\phi_\tau\ge 1$. Then the following hold:
\begin{enumerate}
 \item $\F_I(\phi_\tau)$ is open and $\J_I(\phi_\tau)$ is closed.
 \item $\phi_\tau^{-1}(\F_I(\phi_\tau))=\F_I(\phi_\tau)= \phi_\tau(\F_I(\phi_\tau))$ and  $\phi_\tau^{-1}(\J_I(\phi_\tau))=\J_I(\phi_\tau)= \phi_\tau(\J_I(\phi_\tau))$.
 \item For every integer $m \ge 1$, $\F_I(\phi_\tau^m)= \F_I(\phi_\tau)$ and $\J_I(\phi_\tau^m)= \J_I(\phi_\tau)$.
 \item For any $\eta \in \PGL(2, K)$, if we set $\psi_\tau:= \eta \circ\phi_\tau\circ \eta^{-1}$, then
 \[\F_I(\psi_\tau)= \eta(\F_I(\phi_\tau)) \quad \text{and} \quad \J_I(\psi_\tau)=\eta(\J_I(\phi_\tau)).\]
\end{enumerate}
\end{prop}

\begin{proof}

Statements (1) and (4) are straightforward.  Let us show statements (2) and (3). For statement (2), it suffices to show $\phi_\tau^{-1}(\F_I(\phi_\tau))=\F_I(\phi_\tau)= \phi_\tau(\F_I(\phi_\tau))$. Pick $x\in\phi_\tau^{-1}(\F_I(\phi_\tau))$. Then $\phi_\tau(x)\in\F_I(\phi_\tau)$ and hence $\{\phi_\tau^n\}_{n\ge 1}$ is equicontinuous on a neighborhood $U\subset\bP^1$ of $\phi_\tau(x)$. Take a neighborhood $V$ of $x$ contained in the preimage $\phi_\tau^{-1}(U)$. We conclude by Lemma \ref{lem:holder} that $\{\phi_\tau^n\}_{n\ge1}$ is equicontinuous on $V$, and $x\in\F_I(\phi_\tau)$. Now we pick $y\in\F_I(\phi_\tau)$ and let $w=\phi_\tau(y)$. We assert that $w\in\F_I(\phi_\tau)$. Indeed, since $y\in\F_I(\phi_\tau)$, there exists a neighborhood $W\subset\bP^1$ of $y$ such that $\{\phi_\tau^n\}_{n\ge 1}$ is equicontinuous on $W$. As $\phi_\tau$ is an open map by Lemma \ref{cor:open}, $\phi_\tau(W)$ is a neighborhood of $w$ on which $\{\phi_\tau^{n-1}\}_{n\ge 1}$ is equicontinuous. Thus $w\in\F_I(\phi_\tau)$.  

Now let us prove statement (3). It suffices to show $\F_I(\phi_\tau^m)= \F_I(\phi_\tau)$. We have $\F_I(\phi_\tau)\subset\F_I(\phi_\tau^m)$ because $\{(\phi_\tau^m)^n\}_{n\ge 1}$ is a subsequence of $\{\phi_\tau^n\}_{n\ge 1}$.  For the inverse inclusion, pick $x\in\F_I(\phi_\tau^m)$. Then for any $j\ge 1$, there exists $n\ge 0$ such that $j=mn+i$ for some $0\le i\le m-1$. It follows from Lemma \ref{lem:holder} that there exists $C\ge1$ such that for any $x_1,x_2\in\bP^1$,
$$\sigma(\phi_\tau^j(x_1),\phi_\tau^j(x_2))\le C\sigma(\phi_\tau^{mn}(x_1),\phi_\tau^{mn}(x_2))^{\lambda_\tau^i}$$
Now we choose $x_1, x_2$ close to $x$. Since $x\in\F_I(\phi_\tau^m)$, the sequence $\{(\phi_\tau^m)^n\}_{n\ge 1}$ is equicontinuous at $x$. If $\lambda_\tau\ge 1$, from the above inequality, we immediately obtain that $\{\phi_\tau^j\}_{j\ge 1}$ is equicontinuous at $x$. If  $0<\lambda_\tau<1$, again from the above inequality, we have
$$\sigma(\phi_\tau^j(x_1),\phi_\tau^j(x_2))\le C\sigma(\phi_\tau^{nm}(x_1),\phi_\tau^{nm}(x_2))^{\lambda_\tau^m}.$$
Since $0<\lambda_\tau^m<1$ is a constant, we conclude that $\{\phi_\tau^j\}_{j\ge 1}$ is equicontinuous at $x$. Thus $x\in\F_I(\phi_\tau)$.
\end{proof}

\subsection{Comparison of classical and Berkovich Fatou sets}

The following is an analogue of the ``only if" part of \cite[Theorem 5.19]{Ben19}. 
\begin{lem}\label{lem:rigidF}
Let $\phi_\tau$ be a twisted rational map with $\deg\phi_\tau\ge 1$ and pick $x\in\F_I(\phi_\tau)$. Then there exists a disk $D\subset\bP^1$ containing $x$ such that the set $\bP^1\setminus\cup_{n\ge 0}\phi_\tau^n(D)$ is infinite. 
\end{lem}
\begin{proof}
Since the iterates of $\phi_\tau$ are equicontinuous at $x\in\F_I(\phi_\tau)$, there exists a $\delta>0$ such that for all $y\in D(x,\delta)\subset\bP^1$ and for all $n\geq 0$, 
$$\sigma(\phi_\tau^n(y),\phi_\tau^n(x))<1.$$

{\bf Case 1}: There exists $0<\delta'<\delta$ such that $D(x,\delta')\cap \phi_\tau^n(D(x,\delta'))=\emptyset$ for all $n\ge 1$. Then for any $\delta''$ such that $0<\delta''<\delta'$, the set $\bP^1\setminus\cup_{n\ge 0}\phi_\tau^n(D(x,\delta''))$ contains $D(x,\delta')\setminus D(x,\delta'')$. Therefore $D=D(x,\delta'')$ is a disk as required.

{\bf Case 2}: There exist $0<\delta'<\delta$ and $m\ge 1$ such that $\phi_\tau^m(D(x,\delta'))\subseteq D(x,\delta')$. It follows that $\cup_{n\ge 0}\phi_\tau^n(D(x,\delta'))\subset \bP^1\subset \sP^1$ is contained in at most $m$ directions at the Gauss point $\xi_G$. Consequently $\bP^1\setminus\cup_{n\ge 0}\phi_\tau^n(D(x,\delta'))$ contains all classical points in some direction at $\xi_G$ because there are infinitely many directions at $\xi_G$. We obtain the conclusion by taking $D=D(x,\delta')$.

{\bf Case 3}: For any $0<\delta'<\delta$, there exists $m:=m(\delta')\ge 1$ such that $D(x,\delta')\subsetneq \phi_\tau^m(D(x,\delta'))$. Let us rule out this case by showing that $\{\phi_\tau^n\}_{n\ge 1}$ is not equicontinuous at $x$ which contradicts the assumption that $x\in\F_I(\phi_\tau)$. For any $0<\delta'<\delta$ and the corresponding $m$ such that $D(x,\delta')\subsetneq \phi_\tau^m(D(x,\delta'))$,  since $x\in \phi_\tau^m(D(x,\delta'))\cap D(x,\delta)$, we have either $D(x,\delta)\subseteq\phi_\tau^m(D(x,\delta'))$ or $\phi_\tau^m(D(x,\delta'))\subsetneq D(x,\delta)$. If there exists a sequence $\delta'_i\to 0$ such that $D(x,\delta)\subseteq\phi_\tau^m(D(x,\delta_i'))$, 
 then $\{\phi_\tau^n\}_{n\ge 1}$ is not equicontinuous at $x$, and we are done. Now we assume that $\phi_\tau^m(D(x,\delta'))\subsetneq D(x,\delta)$ for all sufficiently small $\delta'$. Let $D'=D'(\delta')$ be the smallest disk contained in $D(x,\delta)$ that contains all forward iterated images $\phi_\tau^j(D(x,\delta'))$ with $x\in \phi_\tau^j(D(x,\delta'))\subset D(x,\delta)$. It follows that $D'$ is not empty.

Let us prove that there exists $\ell\ge 0$ such that $D(x,\delta)\subseteq\phi_\tau^{\ell}(D')$. If $D'=D(x,\delta)$ then we can simply take $\ell=0$. Now assume that $D'\subsetneq D(x,\delta)$. By the assumption in this case, there exists $m'\ge 1$ such that $D'\subsetneq\phi_\tau^{m'}(D')$. Since by construction $D'$ is exhausted by disks of the form $\phi_\tau^j(D(x,\delta'))$, there exists $j_0\ge 1$ such that $D'\subsetneq\phi_\tau^{j_0+m'}(D(x,\delta'))\subseteq\phi_\tau^{m'}(D')$. Then by definition of $D'$, the disk $\phi_\tau^{j_0+m'}(D(x,\delta'))$ is not properly contained in $D(x,\delta)$. As both disks $\phi_\tau^{j_0+m'}(D(x,\delta'))$ and $D(x,\delta)$ contain $x$, we necessarily have $D(x,\delta)\subseteq\phi_\tau^{j_0+m'}(D')$.

We have proved $D(x,\delta)\subseteq\phi_\tau^{\ell}(D')$. Again as $D'$ is exhausted by disks of the form $\phi_\tau^j(D(x,\delta'))$ with $j\geq 1$, there exists $i\ge 1$ such that $D(x,\delta/2)\subsetneq\phi_\tau^i(D(x,\delta'))$. Since $\delta'$ is arbitrarily small, we conclude that $\{\phi_\tau^n\}_{n\ge 1}$ is not equicontinuous at $x$
\end{proof}

\begin{rmk}
For a nonconstant rational map $\phi\in K(z)$, a point $x$ is in the classical Fatou set $\F_I(\phi)$ if and only if there exists a disk $D\subset\bP^1$ containing $x$ such that $\bP^1\setminus\cup_{n\ge 0}\phi^n(D)$ contains at least two elements (see \cite[Theorem 5.19]{Ben19}). For a twisted rational map $\phi_\tau$, the above lemma asserts that one implication still holds. In fact, applying a similar argument as in the ``if" part of \cite[Theorem 5.19]{Ben19}, we can also prove that the other implication holds in the case where $\lambda_\tau\ge 1$. However, if  $0<\lambda_\tau<1$, then the other implication does not hold anymore for general $\phi_\tau$. The reason is that in contrast to rational maps, the non-archimedean Montel's theorem (\cite[Theorem 2.2]{Hsia00}) fails in this case because such a $\tau$ expands locally the spherical metric (see Section \ref{ex:monomial}).
\end{rmk}

Now we can prove Theorem \ref{thm:F-J}, which is a counterpart of \cite[Theorem 8.3]{Ben19} for twisted rational maps.
\begin{proof}[Proof of Theorem \ref{thm:F-J}]
Pick $x\in\F_I(\phi_\tau)$. We need to show that $x\in\F(\phi_\tau)\cap\P^1$. By Lemma \ref{lem:rigidF}, we can consider a disk $D\subset\bP^1$ containing $x$ such that the complement of $\cup_{n\ge 0}\phi_\tau^n(D)$ contains infinitely many elements. Let $\aD\subset\sP^1$ be the Berkovich disk such that $\aD\cap\bP^1=D$. It follows that 
$$\bigcup_{n\ge 0}\phi_\tau^n(D)=\left(\bigcup_{n\ge 0}\phi_\tau^n(\aD)\right)\cap\sP^1$$
 since $\phi_\tau^n(D)=\phi_\tau^n(\aD)\cap\bP^1$ by \cite[Theorem 7.8]{Ben19} and Lemma \ref{lem:fund1}. Thus $\cup_{n\ge 0}\phi_\tau^n(\aD)$ omits infinitely many points, and hence $x\in\F(\phi_\tau)\cap\bP^1$. Therefore the first assertion of Theorem \ref{thm:F-J} holds.

Now let us assume $\lambda_\tau\ge 1$ and pick $x\in\F(\phi_\tau)\cap\bP^1$. We show that $x\in\F_I(\phi_\tau)$, which proves the second assertion of Theorem \ref{thm:F-J}.  
  Note that by definition of $\F(\phi_\tau)$, there exists a Berkovich open set $U$ containing $x$ such that $X:=\cup_{n\ge0}\phi_\tau^n(U)$ omits infinitely many points in $\sP$. As $x\in\bP^1$, we may assume that $U$ is a Berkovich disk up to shrinking $U$. It follows from Proposition \ref{prop:image-whole} that $\phi_\tau^n(U)$ is also a Berkovich disk for all $n\ge 1$ since $X\neq\sP^1$. 

Now pick two distinct points $x_1,x_2\in\sP^1\setminus X$. We assert that $[x_1,x_2]\subset\sP^1\setminus X$. Assume on the contrary that there is a point $x_3\in ]x_1,x_2[\cap X$. Then $x_3$ is contained in some disk $\phi_\tau^n(U)$. It follows that $\phi_\tau^n(U)$ contains either $x_1$ or $x_2$ since $x_1$ and $x_2$ belong to two different directions at $x_3$. This contradicts the choice of $x_1$ and $x_2$. 

Consider any type II point $x_0\in ]x_1,x_2[$. By the previous paragraph, each disk $\phi_\tau^n(U)$ is contained in one direction at $x_0$. We choose an $M\in\PGL(2,K)$ such that $M(\xi_G)=x_0$ and $U\subset M(\aD(0,1))$. We set $\Phi_n:=M^{-1}\circ\phi_\tau^n\circ M$, for all $n\ge 0$. Then for a fixed $n\ge 0$ and for all points $y\in M^{-1}(U)\cap\bP^1$, the images $\Phi_n(y)$ are contained in the same direction at $\xi_G$. 
Writing $\Phi_n$ as $\tau^n\circ h_n$ for some rational map $h_n\in K(z)$ by Lemma \ref{cor:composition}, we claim that  $\{h_n\}_{n\ge 0}$ is equicontinuous on  $M^{-1}(U)\cap \bP^1$. Indeed, denoting by $W_n=\Phi_n(M^{-1}(U))$, for each $n\ge 0$, we have that $W_n$ is contained in a single direction at $\xi_G$, so  $\tau^{-n}(W_n)$ is contained in a single direction at $\xi_G$ by Lemma \ref{lem:fund} (2), which implies that $h_n(M^{-1}(U))=\tau^{-n}(W_n)$ is contained in a single direction at $\xi_G$ and hence $\{h_n\}_{n\ge 1}$ is equicontinuous on  $M^{-1}(U)\cap \bP^1$ by \cite[Lemma 5.8]{Ben19}. Then applying Lemma \ref{lem:holder}, we conclude that $\{\Phi_n=\tau^n\circ h_n\}_{n\ge 0}$ is equicontinuous on $M^{-1}(U)\cap \bP^1$, since the hypothesis $\lambda_\tau\ge 1$ implies  $\lambda_{\tau^n}=(\lambda_\tau)^n\ge 1$. It follows that $x\in\F_I(\phi_\tau)$. 
\end{proof}

\begin{rmk}
There exists $\phi_\tau$ with $0<\lambda_\tau<1$ such that $\F_I(\phi_\tau)=\emptyset$ but $\F(\phi_\tau)\cap\bP^1=\bP^1$, see Section \ref{ex:monomial}.
\end{rmk}

\subsection{Periodic points}
\footnote{We thank Richard Birkett for pointing out to us type III repelling fixed points and for related discussions.} We classify periodic points of twisted rational maps in this subsection.  A point $\zeta\in\sP^1$ is fixed by a twisted rational map $\phi_\tau$ if $\phi_\tau(\zeta)=\zeta$. Let us begin with type I fixed points.

\begin{defn}
Let $\phi_\tau$ be a twisted rational map with $\deg\phi_\tau\ge 1$. We say that a fixed point $\zeta\in\bP^1$ of $\phi_\tau$ is 
\begin{enumerate}
\item \emph{indifferent} if $\phi_\tau(D)=D$ for any sufficiently small disk $\zeta\in D\subset\bP^1$; 
\item \emph{attracting} if $\phi_\tau(D)\subsetneq D$ for any sufficiently small disk $\zeta\in D\subset\bP^1$; or
\item  \emph{repelling} if $D\subsetneq\phi_\tau(D)$ for any sufficiently small disk $\zeta\in D\subset\bP^1$.
\end{enumerate}
\end{defn}

For the non type I points, in view of Lemma \ref{lem:fund}, we first need to define the corresponding notions direction-wise. Recall that $m_{\vec{v}}(\phi_\tau)$ is the directional multiplicity of $\phi_\tau$ at a direction $\vec{v}\in T_\xi\sP^1$.

\begin{defn}
Let $\phi_\tau$ be a twisted rational map with $\deg\phi_\tau\ge 1$. We say that a direction $\vec{v}\in T_\zeta\sP^1$ at $\zeta\in\sH^1$ is 
\begin{enumerate}
\item \emph{indifferent} if $\lambda_\tau m_{\vec{v}}(\phi_\tau)=1$,
\item \emph{attracting} if $\lambda_\tau m_{\vec{v}}(\phi_\tau)<1$, 
\item  \emph{repelling} if $\lambda_\tau m_{\vec{v}}(\phi_\tau)>1$.
\end{enumerate}
\end{defn}

 With the aid of directions, we can classify the non type I  fixed points as follows.
\begin{defn}
Let $\phi_\tau$ be a twisted rational map with $\deg\phi_\tau\ge 1$. We say that a fixed point $\zeta\in\sH^1$ of $\phi_\tau$ is 
\begin{enumerate}
\item \emph{indifferent} if all directions in $T_\zeta\sP^1$ are indifferent;
\item \emph{attracting} if at least one direction in $T_\zeta\sP^1$ is attracting and the rest are indifferent; 
\item  \emph{repelling} if at least one direction in $T_\zeta\sP^1$ is repelling and the rest are indifferent; or 
\item  \emph{saddle} if $T_\zeta\sP^1$ contains both attracting and repelling directions.
\end{enumerate}
\end{defn}

A periodic point $\zeta\in\sP^1$ of $\phi_\tau$ with exact period $\ell\ge 1$ is a fixed point of $\phi_\tau^\ell$. The above definitions extend to periodic points.

For rational maps, any non type I periodic point is either indifferent or repelling depending on whether the corresponding local degree is $1$. For twisted rational maps, saddle periodic points do exist, see Section \ref{ex:monomial}. We have the following:

\begin{lem}\label{lem:non-rigid-per}
 Let $\phi_\tau$ be a twisted rational map with $\deg\phi_\tau\ge 1$. Let $\zeta \in \sH^1$ be a periodic point of $\phi_\tau$. Then  the following hold.
 \begin{enumerate}
 \item If $0< \lambda_\tau< 1$, then $\zeta$ is either indifferent, attracting, repelling or saddle.
 \item If $\lambda_\tau = 1$, then $\zeta$ is either indifferent or repelling. 
 \item If $\lambda_\tau> 1$, then  $\zeta$ is repelling.
 \end{enumerate}
\end{lem}

\begin{proof}
Up to taking iteration, we assume that $\zeta$ is a fixed point of $\phi_\tau$. There is nothing to show if $0< \lambda_\tau< 1$.

Suppose that $\lambda_\tau = 1$. If $\deg_\zeta\phi_\tau = 1$, then for every direction $\vec v\in T_\zeta\sP^1$, we have $m_{\vec v}(\phi_\tau) = 1$. It follows that $\zeta$ is indifferent. If $\deg_\zeta\phi_\tau > 1$, then $T_\zeta\sP^1$ contains directions with directional multiplicities at least $2$. This implies that such directions are repelling and that all other directions are indifferent. Thus  statement (2) holds. 

Statement (3) follows from the fact that $m_{\vec v}(\phi_\tau) \ge 1$ and  hence $\lambda_\tau m_{\vec v}(\phi_\tau)  > 1$. 
\end{proof}

\begin{rmk} 
\begin{enumerate} 
 \item  If $1/\lambda_\tau$ is not an integer, then $\phi_\tau$ has no indifferent directions.
 \item Each case in Lemma \ref{lem:non-rigid-per} (1) may occur, see Section \ref{ex:monomial}.
 \item If $\phi_\tau$ has an attracting or saddle periodic point in $\sH^1$, then $0<\lambda_\tau<1$.
 \end{enumerate}
\end{rmk}

Now we state some relations between fixed points and Fatou/Julia sets, which are the counterpart of \cite[Theorems 5.14 and 8.7]{Ben19} for twisted rational maps.
\begin{prop}\label{prop:periodic-J}
Let $\phi_\tau$ be a twisted rational map with $\deg\phi_\tau\ge 1$ and $\lambda:=\lambda_\tau$, and let $\zeta\in\sP^1$ be a fixed point of $\phi_\tau$. Then the following hold: 
\begin{enumerate}
\item If $\zeta\in\bP^1$, then $\zeta\in\J_I(\phi_\tau)$ if and only if $\zeta$ is repelling.
\item Assume that $\lambda\ge 1$ and that $\zeta$ is of type II.
\begin{enumerate}
\item If $\zeta$ is indifferent, then $\zeta\in\J(\phi_\tau)$ if and only if there is a bad direction at $\zeta$ having infinite forward orbit under $T_\zeta\phi_\tau$.
\item If $\zeta$ is repelling, then $\zeta\in\J(\phi_\tau)$.
\end{enumerate}
\item Assume that $\lambda\ge 1$ and that $\zeta$ is of type III.
\begin{enumerate}
\item If $\zeta$ is indifferent, then $\zeta\in\F(\phi_\tau)$.
\item If $\zeta$ is repelling, then $\zeta\in\J(\phi_\tau)$.
\end{enumerate}
\item Assume that $\lambda\ge 1$ and that $\zeta$ is of type IV. Then $\zeta$ is an indifferent point contained in $\F(\phi_\tau)$.
\end{enumerate}
\end{prop}

\begin{proof}
Let us first show statement $(1)$.
Consider a small neighborhood $D$ of $\zeta$. Shrinking $D$ if necessary, we may assume that $D\subset\bP^1$ is a disk. If $\zeta\in\bP^1$ is a repelling fixed point, then for any $\xi\in D\setminus\{\zeta\}$, there exists a sequence $n_i\to\infty$ as $i\to\infty$ such that $\phi_\tau^{n_i}(\xi)\notin D$. It follows that $\sigma(\phi_\tau^{n_i}(\zeta), \phi_\tau^{n_i}(\xi))=\sigma(\zeta, \phi_\tau^{n_i}(\xi))\ge\diam(D)$. Hence $\{\phi_\tau^n\}$ is not equicontinuous at $\zeta$. Thus $\zeta\in \J_I(\phi_\tau)$. 

Conversely,  if $\zeta\in\bP^1$ is contained in $\J_I(\phi_\tau)$, then $\{\phi_\tau^n\}_{n\ge 1}$ is not equicontinuous at $\zeta$.  Then there exist $\epsilon>0$ and $\xi\in\bP^1\setminus\{\zeta\}$ arbitrary close to $\zeta$ such that $\sigma(\phi_\tau^{n_0}(\zeta), \phi_\tau^{n_0}(\xi))\ge\epsilon$ for some $n_0\ge 1$. It follows that $\phi_\tau^{n_0}(D(\zeta, \sigma(\zeta,\xi))\not\subset D(\zeta, \sigma(\zeta,\xi))$ if $ \sigma(\zeta,\xi)<\epsilon$. Thus $\zeta$ is neither attracting nor indifferent, and hence is repelling.

Now we begin to show statement $(2)$. By Theorem \ref{thm:F-J} and statement (1), we can assume that $\zeta\in\sH^1$.
Let us first prove statement $(2a)$. Assume that all bad directions $\vec{v}$ at the type II indifferent point $\zeta$ have finite forward orbits. Observing that $\deg_\zeta\phi_\tau=1$ since $\zeta$ is indifferent and $\lambda\ge 1$, we obtain that any bad direction $\vec{v}$ has finite backward orbit as well. Note that $\phi_\tau$ has finitely many bad directions at $\zeta$ by Proposition \ref{prop:multi-basic}. Since $\zeta$ is indifferent, we can remove closed Berkovich disks in these bad directions and their grand orbits under $T_\zeta\phi_\tau$, and obtain an $\phi_\tau$-invariant neighborhood of $\zeta$. Thus $\zeta\in\F(\phi_\tau)$. Conversely, if a bad direction has an infinite forward orbit, then it has infinite backward orbit. Then we obtain that any neighborhood $U$ of $\zeta$ contains an iterated preimage of this bad direction. It follows that some iterate of $U$ is all of $\sP^1$. Thus $x\in\J(\phi_\tau)$.

Let us prove statement $(2b)$. Assume that $\zeta$ is repelling.
Up to changing coordinates, by Lemma \ref{cor:composition}, we can assume that $\zeta=\xi_G$. Let $U\subset\sP^1$ be a Berkovich affinoid containing $\xi_G$. Then there are only finitely many directions at $\xi_G$ not contained in $U$, and hence there are only finitely many directions $\vec{v}$ at $\xi_G$ not contained in any iterate $\phi_\tau^n(U)$. Denote by $S$ the set consisting of the above directions $\vec{v}$. Then every direction in $S$ has finite backward orbit and finite forward orbit under $T_{\xi_G}\phi_\tau$. Considering iterate of $\phi_\tau$ if necessary, for each direction $\vec{v}\in S$, we may assume that  $(T_{\xi_G}\phi_\tau)^{-1}(\vec{v})=\{\vec{v}\}$ and hence $m_{\vec{v}}(\phi_\tau)=\deg T_{\xi_G}\phi_\tau$. Now pick any direction $\vec{u}\in S$, and up to conjugacy, we can assume that $\vec{u}$ is the direction at $\xi_G$ containing $0$. Then for any annulus $A\subseteq \aD(0,1)$ with a boundary point $\xi_G$, we have  $A\subseteq\phi_\tau(A)$ because $\zeta$ is repelling. If $\lambda_\tau>1$, by Lemma \ref{thm:imgtypeII}, we obtain that $\aD(0,1)\setminus\cup_{n\ge0}\phi_\tau^n(A)$ contains at most one point $0$.  If $\lambda_\tau=1$, then $\xi_G\in\aR(\phi_\tau)$ since $\xi_G$ is repelling. In this case, we first observe that $m_{\vec{u}}(\phi_\tau)=\deg T_{\xi_G}\phi_\tau\ge 2$. Then applying Lemma \ref{thm:imgtypeII}, we deduce that $\aD(0,1)\setminus\cup_{n\ge0}\phi_\tau^n(A)$ contains at most one point $0$. In both cases, we conclude that $\cup_{n\ge 1}\phi_\tau^n(U)$ omits at most finitely many points. Thus $\zeta\in\J(\phi_\tau)$.

Now we prove statement (3), up to conjugacy, we can assume that $\zeta\in]0,\xi_G[$.  For statement (3a), pick $\xi\in]0,\zeta[$ sufficiently close to $\zeta$ and consider the annulus $A\subset\sP^1$ with $\partial A=\{\zeta, \xi\}$. Since $\zeta$ is indifferent we have $\deg_\zeta\phi_\tau=\lambda=1$ and we conclude by Lemma \ref{thm:imgtypeII}  that $\phi_\tau^2(A)=A$. It follows that $T_\zeta\phi_\tau^2$ fixes each direction at $\zeta$. To see $\zeta\in \F(\phi_\tau)$, we consider $\xi'\in(\zeta,\infty)$ sufficiently close to $\zeta$ and consider the annulus $A'\subset\sP^1$ with $\partial A'=\{\zeta, \xi'\}$. Applying Lemma \ref{thm:imgtypeII} again, we conclude that $\phi_\tau^2(A')=A'$. Thus the neighborhood $A\cup A'$ is fixed by $\phi_\tau^2$, so $\zeta\in\F(\phi_\tau^2)=\F(\phi_\tau)$ (see Proposition \ref{prop:fatoujuliabasics} (3)).  For statement (3b), noting that $T_\zeta\sP^1$ contains exactly two directions, we obtain the conclusion by a similar argument as in statement (2b).

\smallskip
Let us show statement $(4)$. Suppose on the contrary that $\zeta$ is not indifferent. Then by Lemma \ref{lem:non-rigid-per} and by the assumption that $\lambda\ge 1$, we have that $\zeta$ is repelling. Then there exists a type II point $\zeta_1\in]\zeta,\infty[$ sufficiently close to $\zeta$ such that $\zeta_2:=\phi_\tau(\zeta_1)\in]\zeta_1,\infty[$. We can further assume that $\phi_\tau$ has constant local degree on $[\zeta, \zeta_2]$ and $\phi_\tau$ maps the Berkovich annulus $A$ with boundary $\zeta$ and $\zeta_2$ not to all of $\sP^1$. Change coordinates so that $\zeta_1,\zeta_2\in ]0,\xi_G[$ and $\phi_\tau(0)\not=0$. Consider the corresponding closed disks $\overline{D}_1$ and $\overline{D}_2$ in $K$ for $\zeta_1$ and $\zeta_2$, respectively. It follows that $\overline{D}_1\subsetneq\overline{D}_2$. 
 By \cite[Proposition 4.17]{Ben19} the rational map $\phi$ has a fixed point in $\tau(\overline{D}_1)$. Let this fixed point be $\tau(a)\in\tau(\overline{D}_1)$ for some $a\not=0$. Set $M(z)=z+a-\tau(a)$. It follows that $M\circ\phi_\tau$ fixes $a$. Note that $\tau(\zeta_1)\in [0,\zeta_1]$ because $\lambda\geq 1$, which implies that $a$ and $\tau(a)$ are both contained in $\overline{D}_1$ and hence $M$ fixes $\overline{D}_1$. Denote by $\xi:=a\vee\zeta\in\sH^1$. Then 
 $\zeta_1\in [\xi,\infty]$. Since $\zeta$ is repelling and $M\circ\phi_\tau$ has constant local degree on $[\zeta,\zeta_2]$, the direction at $\xi$ containing $a$ maps to a direction at $M\circ\phi_\tau(\xi)\in]\xi,\zeta_2]$ not containing $\xi$ under $T_\xi(M\circ\phi_\tau)$. Since $\phi_\tau(A)\not=\sP^1$, we conclude that  $a$ cannot be fixed by $M\circ\phi_\tau$. This is a contradiction. Thus $\zeta$ is indifferent.  Moreover, applying again the argument used in statement (3a), we can conclude that $\zeta\in\F(\phi_\tau)$. 
\end{proof}

When $\lambda_\tau\ge 1$, the above result on periodic points of type I, II or IV is the same as in the rational case, see \cite[Theorem 8.7]{Ben19}. Under a mild assumption on the field $K$ 
 or on $\tau$, we can rule out type III repelling fixed points for twisted rational maps:

\begin{prop}\label{prop:periodic-III}
Let $\phi_\tau$ be a twisted rational map with $\deg\phi\ge 1$. Assume  $\lambda:=\lambda_\tau\ge 1$ and let $\zeta\in\sP^1$ be a type III fixed point of $\phi_\tau$. If $\lambda$ is rational or if $\log|K^\times|$ is a field, then $\zeta$ is an indifferent fixed point contained in $\F(\phi_\tau)$.
\end{prop}
\begin{proof}
Up to conjugacy, we may assume that $\zeta\in]0,\xi_G[$. By Lemma  \ref{thm:imgtypeII}, there exist $d\ge 1$ and $b\in K$ such that 
\[
\zeta=\phi_\tau(\zeta)=\zeta_{\phi_\tau(0), |b|\diam(\zeta)^{d\lambda}}.
\]
In particular $\diam(\zeta)=|b|\diam(\zeta)^{d\lambda}$. Thus $\diam(\zeta)^{1-d\lambda}=|b|$ is in $|K^\times|$. 

Assume by contradiction that $1-d\lambda\not=0$. We have $\diam(\zeta)=|b|^{1/(1-d\lambda)}$. If $\lambda\in \bQ$, then $|b|^{1/(1-d\lambda)}\in |K^\times|$ because $K$ is algebraically closed. If $\log|K^\times|$ is a field, we also have $|b|^{1/(1-d\lambda)}\in |K^\times|$ because $\lambda\in\log |K^\times|$ and hence $1/(1-d\lambda)\in\log |K^\times|$.
 In either case we have $\diam(\zeta)\in |K^\times|$, which contradicts that $\zeta$ is of type III. Thus $d\lambda-1=0$. Since $\lambda\ge 1$, this implies that $d=\lambda=1$. Moreover, since $\zeta$ is of type III, we have $\deg_\zeta\phi_\tau=d=1$, i.e.\ $\zeta$ is indifferent. Thus by Proposition \ref{prop:periodic-J} (3a), we have $\zeta\in\F(\phi_\tau)$.
\end{proof}

\begin{rmk}\label{fieldvaluationcondition}
If $K$ has a discrete valued subfield whose algebraic closure is dense in $K$, then 
$K$ satisfies the assumption in Proposition \ref{prop:periodic-III}.
\end{rmk}

The following is an immediate consequence of Proposition \ref{prop:periodic-J}. 
\begin{cor}\label{cor:F}
Let $\phi_\tau$ be a twisted rational map with $\deg\phi_\tau\ge 1$. Assume $\lambda_\tau\ge 1$ and let $\zeta\in\sH^1$ be a fixed point of $\phi_\tau$ in $\F(\phi_\tau)$. If a direction $\vec{v}\in T_\zeta\sP^1$ has infinite forward orbit under $T_\zeta\phi_\tau$, then $\vec{v}$ is contained in $\F(\phi_\tau)$.
\end{cor}
\begin{proof}
As $\vec{v}$ has infinite forward orbit, the fixed point $\zeta$ is of type II. By Lemma \ref{lem:non-rigid-per}, the fixed point $\zeta$ is either repelling or indifferent. By Proposition \ref{prop:periodic-J} (2), we conclude that $\zeta$ is indifferent and that the direction $\vec{v}$ is good. The conclusion follows.
\end{proof}

We say that the twisted rational map $\phi_\tau$ is \emph{simple} if $\J(\phi_\tau)$ is a singleton; otherwise, we say that $\phi_\tau$ is \emph{nonsimple}. Proposition \ref{prop:periodic-J} implies the following:

\begin{cor}\label{cor:simple}
Let $\phi_\tau$ be a simple twisted rational map with $\deg\phi_\tau\ge 2$ and $\lambda_\tau\ge 1$.
Assume that $\phi_\tau$ has no type III repelling fixed points.  
Then $\J(\phi_\tau)$ consists of a unique type II point.
\end{cor}
\begin{proof}
Write $\J(\phi_\tau)=\{\zeta\}$. Then $\zeta$ is a fixed point with $\deg_\zeta\phi_\tau=\deg \phi_\tau\ge 2$. It follows that $\zeta$ is repelling. By Proposition \ref{prop:periodic-J} we conclude that $\zeta$ is of type II. 
\end{proof}


\begin{prop}\label{prop:repelling}
Let $\phi_\tau$ be a twisted rational map with $\deg\phi_\tau\ge 2$. Assume $\lambda:=\lambda_\tau\ge 1$. Then 
\begin{enumerate}
\item $\phi_\tau$ has a repelling fixed point in $\sP^1$ and hence $\J(\phi_\tau)\not=\emptyset$.
\item Both $\F(\phi_\tau)$ and $\F_I(\phi_\tau)$ are non-empty. Moreover, if $\lambda>1$, then $\F_I(\phi_\tau)=\bP^1$. 
\end{enumerate}
\end{prop}

\begin{proof}
The proof of \cite[Theorem 12.5]{Ben19} can also be applied to prove statement $(1)$. 

We now prove statement (2). Let us first show that $\F(\phi_\tau)\not=\emptyset$. If $\J_I(\phi_\tau)=\emptyset$, then $\F_I(\phi_\tau)\not=\emptyset$ and thus $\F(\phi_\tau)\not=\emptyset$ by Theorem \ref{thm:F-J}. We now assume that $\J_I(\phi_\tau)\not=\emptyset$. Let $z_0\in J_I(f)$ and consider $Z:=\bigcup_{n=0}^\infty\phi_\tau^{-n}(\{z_0\})$. Then $Z$ is a countable set. Let $U$ be a small $\rho$-neighborhood  of $\xi_G$ in $\sH^1$. Then $U$ contains uncountably many type II points. Thus there exists a point $\xi\in U$ and $\vec{v}\in T_{\xi}\sP^1$ such that $\sB(\vec{v})\cap Z=\emptyset$. It follows that $z_0\not\in\cup_{n=0}^\infty\phi_\tau^n(\sB(\vec{v}))$ and hence $\sB(\vec{v})\subset\F(\phi_\tau)$. Then by Theorem \ref{thm:F-J}, we also have $\F_I(\phi_\tau)=\emptyset$.

The second assertion in statement $(2)$ is an direct application of Lemma \ref{lem:holder}.
 Indeed, considering the constant $C\ge 1$ for $\phi_\tau$ as in Lemma \ref{lem:holder}, we have that for any $x,y\in\bP^1$
$$\sigma(\phi_\tau^n(x),\phi_\tau^n(y))\le C^{1+\lambda+\cdots+\lambda^{n-1}}\sigma(x,y)^{\lambda^n}.$$
Then the assumption $\lambda>1$ implies that $\{\phi_\tau^n\}_{n\ge 1}$ is equicontinuous at any point $x\in\bP^1$. Hence in this case we have $\F_I(\phi_\tau)=\bP^1$.
\end{proof}

\begin{rmk}\label{Fatouempty}
For rational maps, the nonemptyness of classical Fatou set can be deduced from the existence of a nonrepelling fixed point, see \cite[Proposition 4.2]{Ben19}, which is an application of a nonarchimedean version of Holomorphic Fixed-Point Formula, see \cite[Proposition 1.2]{Benedetto01}. We do not expect that such a fixed-point formula holds in our twisted rational map case due to Remark \ref{rmk:not}. So it is unclear to us whether a twisted rational map always has type I nonrepelling fixed point. 
\end{rmk}

We emphasize that when $0<\lambda_\tau<1$, the Berkovich Fatou set $\F(\phi_\tau)$ may contain type I repelling fixed points and the classical Fatou set $\F_I(\phi_\tau)$ may be empty, see Section \ref{ex:monomial}. 


\section{Twisted polynomials}
In this section we first treat the example of twisted monimials and then apply Trucco's method to twisted polynomials. The two subsections are independent.

\subsection{Dynamics of twisted monomials}\label{ex:monomial}

Consider the monomial $f(z)=z^d$ in $K(z)$ with $d\ge 2$. Then the Berkovich Julia set $\J(f)$ is the singleton $\{\xi_G\}$, and $0,\infty$ are (super)attracting points in the classical Fatou set $\F_I(f)\subset\F(f)$. Now pick $\tau\in\Aut^*(K)$ and consider the twisted rational map $f_\tau$. The Gauss point $\xi_G$ is fixed by $f_\tau$. 
Denoting by $\widetilde{K}$ the residue field of $K$ and identifying  $T_{\xi_G}\sP^1$ with $\bP^1(\widetilde{K})$, we have $T_{\xi_G}f_\tau(w)=w^d$ for $w\in\bP^1(\widetilde{K})$.

Direct computations show that for any $x\in]0,\xi_G[$ and any $y\in]\xi_G,\infty[$, we have:
\begin{enumerate}
\item If $0<\lambda_\tau<1/d$, then $f_\tau(x)\in]x,\xi_G[$ and $f_\tau(y)\in]\xi_G,y[$. 
\item If $\lambda_\tau=1/d$, then $f_\tau(x)=x$ and $f_\tau(y)=y$.
\item If $\lambda_\tau>1/d$, then $f_\tau(x)\in]0,x[$ and $f_\tau(y)\in]y,\infty[$.
\end{enumerate}

For the two fixed points $0$ and $\infty$, computing the orbit of a small disk around each of these points, we have: 

\begin{enumerate}
\item If $0<\lambda_\tau<1/d$, then $0$ and $\infty$ are repelling fixed points and are contained in both $\J_I(f_\tau)$ and $\F(f_\tau)$.
\item If $\lambda_\tau=1/d$, then $0$ and $\infty$ are indifferent fixed points and are contained in both $\F_I(f_\tau)$ and $\F(f_\tau)$.
\item If $\lambda_\tau>1/d$, then $0$ and $\infty$ are (super)attracting fixed points and are contained in both $\F_I(f_\tau)$ and $\F(f_\tau)$.
\end{enumerate}

Let us assume that the residue characteristic of $K$ does not divide $d$. Then $f_\tau$ is tame and the ramification locus of $f_\tau$ is $\aR(f_\tau)=[0,\infty]$.
Now pick a Berkovich open disk $\aD\subset\sP^1$. We consider the forward orbit of $\aD$ under $f_\tau$. If $\xi_G\not\in\aD$, then $f_\tau^n(\aD)$ is a disk not containing $\xi_G$ and hence $\aD\subset F(f_\tau)$. In the case where $\xi_G\in\aD$, we have the following:
\begin{enumerate}
\item If $0<\lambda_\tau<1/d$, then $\cup_{n\ge 0}f_\tau^n(\aD)$ omits a closed ball provided that $\aD$ omits $0$ or $\infty$, and hence $\xi_G\in\F(f_\tau)$. In fact, $\xi_G$ is an attracting fixed point as all directions at $\xi_G$ are attracting.
\item If $\lambda_\tau=1/d$, then $\cup_{n\ge 0}f_\tau^n(\aD)$ omits a closed ball provided that $\aD$ omits $0$ or $\infty$, and hence $\xi_G\in\F(f_\tau)$. In fact, $\xi_G$ is an attracting fixed point as at $\xi_G$ the directions containing $0$ and $\infty$ are indifferent and all other directions are attracting.
\item If $\lambda_\tau>1/d$, then $\cup_{n\ge 0}f_\tau^n(\aD)\supseteq\sP^1\setminus\{0,\infty\}$ and hence $\xi_G\in\J(f_\tau)$. In fact, if $1/d<\lambda_\tau<1$, then $\xi_G$ is a saddle fixed point as at $\xi_G$ the directions containing $0$ and $\infty$ are repelling and all other directions are attracting; and if $\lambda_\tau>1$, then $\xi_G$ is a repelling fixed point as at $\xi_G$ all directions are repelling.
\end{enumerate}

In the case where $0<\lambda_\tau<1/d$, we actually have $\J_I(f_\tau)=\bP^1$ while $\F(f_\tau)=\sP^1$.

If the residue characteristic of $K$ divides $d$, then $f_\tau$ is not tame and the segment $[0,\infty]$ is a proper subset of $\aR(f_\tau)$. In fact, for any $\zeta\in[0,\infty]$ and any direction $\vec{v}\in T_\zeta\sP^1$, we have $m_{\vec{v}}(f_\tau)=d$. In this case, we have the following for the fixed point $\xi_G$:

\begin{enumerate}
\item If $0<\lambda_\tau<1/d$ then $\xi_G$ is an attracting fixed point as all directions at $\xi_G$ are attracting.
\item If $\lambda_\tau=1/d$ then $\xi_G$ is an indifferent fixed point as all directions at $\xi_G$ are indifferent.
\item If $\lambda_\tau>1/d$ then $\xi_G$ is a repelling fixed point as all directions at $\xi_G$ are repelling.
\end{enumerate}

\subsection{Dynamics of twisted polynomials}
We will prove Theorem \ref{thm:poly} in this subsection. Let $P\in K[z]$ be a tame polynomial of degree at least $2$ and let $\tau\in\Aut^*(K)$ with $\lambda:=\lambda_\tau>1/\deg P$. We consider the twisted polynomial $P_\tau$. As in \cite{Trucco14}, the dynamics on the Julia set of a polynomial can be described by the so-called Trucco's tree, we extend this description to the above $P_\tau$ provided that $P_\tau$ has no type III repelling fixed point. 

Note that $\infty$ is a superattracting fixed point for $P_\tau$. The \emph{basin of $\infty$} for $P_\tau$ is 
\[\Omega_\infty(P_\tau):=\{ \zeta\in\sP^1: P_\tau^n(\zeta)\rightarrow\infty \} .\]
Observe that $\Omega_\infty(P_\tau) \neq\emptyset$ and contains a neighborhood of $\infty$. We define the \emph{filled Julia set} to be $\K (P_\tau):=\sP^1\setminus \Omega_\infty(P_\tau)$. 
Note that $\J(P_\tau)\subset\K(P_\tau)$. Therefore $\K(P_\tau)\neq\emptyset$ as $\J(P_\tau)\neq\emptyset$ by Proposition \ref{prop:repelling} (1). Moreover, it follows easily that $\J(P_\tau)=\partial \K(P_\tau)=\partial\Omega_\infty(P_\tau)$. 

We will repeatedly use the fact that for any closed Berkovich disk $\overline{\aD}\subsetneq\sP^1\setminus\{\infty\}$, the image $P_\tau(\overline{\aD})\not=\sP^1$ is also a closed Berkovich disk.  Applying an analogue of \cite[Lemma 2.5]{Trucco14}, we have the following. 

\begin{lem}\label{lem:JF}
If $\zeta\in\J(P_\tau)\cap\sH^1$, then for any $\xi\in\sP^1$ with $\zeta\in]\xi,\infty[$, we have $\xi\in\F(P_\tau)\cap\K(P_\tau)$ and $]\zeta,+\infty]\subset\Omega_\infty(P_\tau)$.
\end{lem}

Now we extend the method of Trucco's tree in \cite{Trucco14} to the twisted polynomial $P_\tau$. Let $\overline{\aD}_{min}\subset\sP^1$ be the smallest closed Berkovich disk containing $\K(P_\tau)$. We call the boundary point $\zeta_{P_\tau}\in\sH^1$ of $\overline{\aD}_{min}$ the \emph{base point} of $P_\tau$.  Preimages of $\zeta_{P_\tau}$ satisfy the following properties which will allow us to do the construction of the Trucco's tree. Recall that by definition $P_\tau$ is nonsimple if $\J(P_\tau)$ is not a singleton.

\begin{prop}\label{prop:L}
Suppose that $P_\tau$ is nonsimple. Then the following hold.
\begin{enumerate}
\item $\{\zeta_{P_\tau}\}=P_\tau^{-1}(P_\tau(\zeta_{P_\tau}))$.
\item $P_\tau(\zeta_{P_\tau})\in]\zeta_{P_\tau},\infty[$.
\item $P_\tau^n(\zeta_{P_\tau})\to\infty$, as $n\to\infty$.
\item $\zeta_{P_\tau}$ is of type II. 
\item $P_\tau^{-1}(\zeta_{P_\tau})$ contains at least $2$ elements; moreover, for any two different points $\zeta_1,\zeta_2\in P_\tau^{-1}(\zeta_{P_\tau})$, we have $\zeta_2\not\in]\zeta_1,\zeta_{P_\tau}[$.
\item $P_\tau^{-1}(\zeta_{P_\tau})$ contains points in at least two directions at $\zeta_{P_\tau}$.
\end{enumerate}
\end{prop}

\begin{proof}
All statements except statements (3) and (4) can be obtained by similar arguments as in \cite[Proposition 3.4]{Trucco14}. 



For statement $(3)$, observe that for all $n\ge 0$, 
$$\rho(P_\tau^{n+1}(\zeta_{P_\tau}), P_\tau^{n+2}(\zeta_{P_\tau}))=(\lambda\deg P_\tau)\cdot\rho(P_\tau^n(\zeta_{P_\tau}), P_\tau^{n+1}(\zeta_{P_\tau})).$$
Since $P_\tau^n(\zeta_{P_\tau})\in]P_\tau^{n-1}(\zeta_{P_\tau}), \infty[$ by statement (2) for all $n\ge 1$, we conclude that 
$$\rho(\zeta_{P_\tau}, P_\tau^n(\zeta_{P_\tau}))=\left(\sum_{j=0}^{n-1}(\lambda\deg P_\tau)^j\right)\rho(\zeta_{P_\tau}, P_\tau(\zeta_{P_\tau})).$$
Since $\lambda\deg P_\tau>1$, we have $\rho(\zeta_{P_\tau}, P_\tau^n(\zeta_{P_\tau}))\to\infty\ \ \text{as}\ \ n\to\infty$.
Thus statement $(3)$ follows. 

Let us show statement $(4)$. Suppose on the contrary that $\zeta_{P_\tau}$ is not of type II. Then $\zeta_{P_\tau}$ is of type III. By the definition of $\zeta_{P_\tau}$, since $\zeta_{P_\tau}$ is of type III and $\K(P_\tau)$ is closed, we conclude that $\zeta_{P_\tau}\in\partial\K(P_\tau)$. Consequently $\zeta_{P_\tau}\in\J(P_\tau)$, which contradicts statement $(2)$.
\end{proof}

\begin{rmk}
If $P_\tau$ is nonsimple, then Proposition \ref{prop:L} implies that $\zeta_{P_\tau}$ is the smallest point (with respect to the partial order of $\sP^1$) with $\deg_{\zeta_{P_\tau}}P_\tau=\deg P_\tau$ such that $P_\tau(\zeta_{P_\tau})\in[\zeta_{P_\tau}, \infty[$. Let $\zeta_P$ be the base point for the polynomial $P$. It follows that $\zeta_{P_\tau}=\tau^{-1}(\zeta_P)$.
\end{rmk}

Proposition \ref{prop:periodic-J} and Proposition \ref{prop:L} (4) immediately imply the following.
\begin{cor}
If $\zeta_{P_\tau}$ is of type III, then $\J(P_\tau)=\{\zeta_{P_\tau}\}$.
\end{cor}

For $n\ge 0$, define $\LL_n$ to be the finite set $P_\tau^{-n}(\zeta_{P_\tau})$. Following \cite[Definition 3.5]{Trucco14}, we say that a decreasing sequence $\{L_n\}_{n\ge 0}$ of points in $\sP^1$ such that $L_n\in\LL_n$ is a \emph{dynamical sequence} of $P_\tau$. As in \cite[Proposition 3.6]{Trucco14}, we can describe the Julia set $\J(P_\tau)$ as follows.

\begin{prop}
We have 
$$\J(P_\tau)=\left\{\lim_{n\to\infty}L_n: \{L_n\}_{n\ge 0}\ \ \text{is a dynamical sequence of}\ \ P_\tau\right\}.$$
\end{prop}



Now we prove Theorem \ref{thm:poly}.
\begin{proof}[Proof of Theorem \ref{thm:poly}]
Since $P_\tau$ is tame and all its critical points are in $\Omega_\infty(P_\tau)$, we can choose $N\gg1$ such that each $\zeta\in\LL_N$ has $\deg P_\tau$ preimages under $P_\tau$.  Now we associate to each element $\xi\in\LL_{N+1}$ an integer $\chi(\xi)\in\{1, \dots, \deg P_\tau\}$ so that if two distinct points $\xi_1,\xi_2\in\LL_{N+1}$ have the same image $P_\tau(\xi_1)=P_\tau(\xi_2)\in\LL_{N}$, then $\chi(\xi_1)\not=\chi(\xi_2)$. Now define a function 
$$\iota:\J(P_\tau)\to\{1, \dots, \deg P_\tau\}^{\N\cup\{0\}}$$
as follows: if $\{L_n(x)\}_{n\ge 0}$ is a dynamical sequence converging to a point $x\in\J(P_\tau)$, then the image $\iota(x)$ is $(i_0, i_1, \dots)$, where $i_j=\chi(P_\tau^j(L_{N+1+j}(x)))$. Then $\iota$ gives the desired topological conjugacy.
\end{proof}

\begin{rmk}
For a tame polynomial $P$ with only escaping critical points, the Berkovich Julia set $\J(P)$ is contained in $\bP^1$, on which the dynamics is topologically conjugate to the one-sided shift on $d$ symbols, see \cite[Theorem 3.1]{Kiwi06}. However, due to the factor $\lambda\ge 1/\deg P$, in our case, $\J(P_\tau)$ may not be contained in $\bP^1$; in fact, by \eqref{eq:dis}, a direct distance computation shows that $\J(P_\tau)\subset\bP^1$ if and only if $\lambda\ge1$. 
\end{rmk}

Note that we can define the \emph{Trucco's tree} for $P_\tau$ as the following subtree of $\sP^1$:
\[
\T_{P_\tau}:=\mathrm{Hull}\left(\bigcup_{n\ge 0}\LL_n\right). 
\]
The vertices of $\T_{P_\tau}$ are grand orbits of points of valence at least three.


\section{Equidistribution}\label{sec:equ}

In this section we construct a canonical measure for twisted rational maps and establish Theorem \ref{thm:equ}. We apply essentially the same procedure as in \cite{Jon15}. Let $\phi_\tau$ be a twisted rational map of degree  $d\ge 2$.
\subsection{Exceptional set}
A point $x\in\P^1$ is an \emph{exceptional point} of $\phi_\tau$ if the grand orbit of $x$ under $\phi_\tau$ is finite. We denote by $E_{\phi_\tau}\subset\P^1_K$ the set of exceptional points of $\phi_\tau$. 
\begin{lem}\label{lem:exception}
Let $\phi_\tau$ be a twisted rational map of degree at least $2$. Then $\#E_{\phi_\tau}\le 2$. More precisely,

\begin{enumerate}
\item If $\#E_{\phi_\tau}=1$, then $\phi_\tau$ is M\"obius conjugate to $P_\tau$ for some polynomial $P\in K[z]$ of degree $\deg\phi_\tau$.
\item If $\#E_{\phi_\tau}=2$, then $\phi_\tau$ is M\"obius conjugate to $Q_\tau$ for some monomial $Q\in K(z)$ of degree $\pm\deg\phi_\tau$.
\end{enumerate}
\end{lem}

\begin{proof}
Let $x\in E_{\phi_\tau}$. Then the grand orbit of $x$ is finite. It follows that there exists $n\ge1$ such that $\phi_\tau^{-n}(\{x\})=\{x\}$. Hence $\phi_\tau^n$ and $\phi_\tau$ are totally ramified at $x$. As the number of totally ramified points of $\phi_\tau$ is the same as $\phi$ (see Lemma \ref{lem:ram}), we have $\#E_{\phi_\tau}\le 2$.

Assume that $\#E_{\phi_\tau}=1$ and pick $M_1\in\PGL(2,K)$ such that $E_{\phi_\tau}=\{M_1(\infty)\}$. Consider the map $M_1^{-1}\circ \phi_\tau\circ M_1$; we write it as $\psi_\tau$ for some rational map $\psi\in K(z)$ (see Lemma \ref{cor:composition}). Since $\tau(\infty)=\infty$, we conclude that $\infty$ is the unique exceptional point of $\psi$. Then by \cite[Theorem 1.19]{Ben19}, the map $\psi$ is a polynomial of degree $\deg\psi=\deg\phi_\tau$. 

Assume that $\#E_{\phi_\tau}=2$ and pick $M_2\in\PGL(2,K)$ such that $E_{\phi_\tau}=\{M_2(\infty), M_2(0)\}$. Consider the map $M_2^{-1}\circ\Phi\circ M_2$; again we write it as $\psi_\tau$ for some rational map $\psi\in K(z)$. Since $\tau(\infty)=\infty$ and $\tau(0)=0$, we conclude that $\infty$ and $0$ are the only exceptional points of $\psi$. Then \cite[Theorem 1.19]{Ben19} says that $\psi$ is a monomial of degree $\deg\psi=\pm\deg\phi_\tau$. 
\end{proof}

\subsection{Canonical measure}
We first introduce some preliminaries by following the presentation in \cite{Jon15}.

Let $\Gamma\subset\sH^1$ be a finite subtree, and for $\zeta\in\Gamma$, denote by $T_\zeta\Gamma$ the set of directions at $\zeta$ containing points in $\Gamma$. For a function $f:\Gamma\to\R$ and a direction $\vec{v}\in T_\zeta\Gamma$, denote by $D_{\vec{v}}f$ the directional derivative of $f$ in $\vec{v}$. If $f$ has bounded differential variation, then the \emph{Laplacian} of $f$ is 
$$\Delta_\Gamma(f)=\sum_{\zeta\in\Gamma}\left(\sum_{{\vec{v}}\in T_\zeta\Gamma}D_{\vec{v}}(f)\right)\delta_\zeta,$$
where $\delta_\zeta$ is the Dirac measure at $\zeta$.

Let $\nu_0$ be a finite atomic measure on $\Gamma$. Denote by $\SH(\Gamma, \nu_0)$ the set of continuous functions $f:\Gamma\to\R$ that are convex on any segment disjoint from the support of $\nu_0$ and such that, for any $\zeta\in\Gamma$, 
$$\nu_0(\zeta)+\sum_{\vec{v}\in T_\zeta\Gamma}D_{\vec{v}} f\ge 0.$$
Each element in  $\SH(\Gamma, \nu_0)$ is called a \emph{$\nu_0$-subharmonic function}.

By using approximation by finite trees, one can extend both the notions of Laplacian and subharmonic functions to the closure $\overline{U}$ of a domain $U\subset\sP^1$. The corresponding notations are $\Delta_{\overline{U}}(f)$ and $\SH(\overline{U}, \nu_0)$, where $\nu_0$ is a finite atomic measure supported on $\sH^1\cap U$. We write $\Delta_{\sP^1}(f)$ simply as $\Delta(f)$ and denote by $\SH^0(\overline{U}, \nu_0)$ the compact subset consisting of $f\in\SH(\overline{U}, \nu_0)$ for which $\max f=0$.

Denote by $C^0(\overline{U})$ the set of real-valued continuous functions on $\overline{U}$. 
Associated to a twisted rational map $\phi_\tau$ is a push-forward operator on continuous functions:
\[(\phi_\tau)_* H(\zeta)=\sum_{\phi_\tau(\xi)=\zeta}\deg_\zeta(\phi_\tau)H(\xi), \quad \text{for any}\ H\in C^{0}(\sP^1).\]
 The pull-back action of $\phi_\tau$ on Radon measures is defined by duality: for a Radon measure $\nu$ on $\sP^1$ we define $\phi_\tau^*\nu$ by
\[\left\langle \phi_\tau^*\nu,H\right\rangle=\left\langle \nu,(\phi_\tau)_* H\right\rangle.\] 
Note that the pull-back of a Dirac mass $\delta_\zeta$ at $\zeta\in\sP^1$ is 
\begin{align}\label{eq:measure}
\phi_\tau^*\delta_\zeta=\sum_{\phi_\tau(\xi)=\zeta}\deg_\xi\phi_\tau\ \delta_\xi.
\end{align}
It follows from Proposition \ref{prop:deg-basic} that $\phi_\tau^*\delta_\zeta(\sP^1)=d$.

The following result concerns the bull-back of $f\in \SH^0(\sP^1,\nu_0)$ and its Laplacian.
 
\begin{lem}\label{lem:laplacianpullback}
Let $\phi_\tau$  be a twisted rational map of degree at least $1$.
  If $f\in \SH^0(\sP^1,\nu_0)$ for  a finite atomic measure $\nu_0$ supported on $\sH^1$, then $\phi_\tau^* f\in \SH^0(\sP^1,\phi_\tau^*\nu_0)$ and
\[\Delta(\phi_\tau^* f)=\lambda_\tau\phi_\tau^*(\Delta (f)).\]
\end{lem}
\begin{proof}
Applying the argument in \cite[Proposition 4.15]{Jon15}, to obtain the conclusion, we only need to prove that $\Delta(\tau^* f)=\lambda_\tau\tau^*(\Delta(f))$ on any finite subtree. For any finite subtree $\Gamma\subset\sH^1$, if $\vec{v}\in T_\zeta\Gamma$ is a direction at $\zeta\in\Gamma$, letting $\vec{w}$ be the image of $\vec{v}$ under $T_\zeta\tau$, we have $D_{\vec{v}}(\tau^* f))=\lambda_\tau D_{\vec{w}}f$ by Lemma \ref{lem:fund} (1) and hence 
\begin{align*}
\Delta(\tau^* f){\{\zeta\}}=\sum_{\vec{v}\in T_\zeta\Gamma}D_{\vec{v}}(\tau^* f)=\sum_{\vec{v}\in T_\zeta\Gamma,\vec{w}=T_\zeta\tau(\vec{v})}\lambda_\tau D_{\vec{w}}f=\lambda_\tau\tau^*(\Delta f)\{\zeta\}.
\end{align*}
\end{proof}

Now pick $\zeta\in \sH^1$. Since $d^{-1}\phi_\tau ^*\delta_\zeta$ is a probability measure, we have 
\[
d^{-1}\phi_\tau^*\delta_\zeta=\delta_\zeta+\Delta u
\]
for some continuous $\delta_{\zeta}$-subharmonic function $u$, see \cite[Section 2.5.2]{Jon15}. Iterating the above equation by Lemma \ref{lem:laplacianpullback}, we obtain
\[d^{-n}(\phi_\tau^n)^*\delta_{\zeta}=\delta_{\zeta}+\Delta u_n \quad \text{where}\quad 
u_n=\sum_{j=0}^{n-1}(d\lambda_\tau)^{-j} u\circ\phi_\tau^j.
\]
If $d\lambda_\tau >1$ then the sequence $\{u_n\}$ converges uniformly to a continuous $\delta_\zeta$-subharmonic function $u_\infty$. In this case we set 
\[\mu_{\phi_\tau}:=\delta_\zeta+\Delta u_\infty.\]
We call $\mu_{\phi_\tau}$ the \emph{canonical measure} of $\phi_\tau$. Since $u_\infty$ is bounded, the measure $\mu_{\phi_\tau}$ does not charge any classical point, see \cite[Section 2.5.2]{Jon15}.

\subsection{Proof of equidistribution}
We assume that $\lambda_\tau d>1$.

\begin{lem}\label{lem:distanceestimate}
Let $\phi_\tau$ be a twisted rational map of degree $d\ge 1$.
 Then for any $\zeta_0,\zeta\in \sH^1$, as $n\to\infty$,
$$\rho(\phi_\tau^n(\zeta),\zeta_0)=O((d\lambda_\tau)^n).$$
\end{lem}

\begin{proof}
The rational map $\phi$ expands the hyperbolic metric by a factor at most $d$ and $\tau$ expands the hyperbolic metric by $\lambda_\tau$. We have
\begin{align*}
\rho(\phi_\tau^n(\zeta),\zeta_0) & \leq \rho(\phi_\tau^n(\zeta),\phi_\tau^n(\zeta_0))+\rho(\phi_\tau^n(\zeta_0),\zeta_0) \leq (d\lambda_\tau)^n\rho(\zeta,\zeta_0)+ \sum_{j=0}^{n-1}\rho(\phi_\tau^{j+1}(\zeta_0),\phi_\tau^j(\zeta_0))\\
& \leq (d\lambda_\tau)^n\rho(\zeta,\zeta_0)+\sum_{j=0}^{n-1}(d\lambda_\tau)^j \rho(\phi_\tau(\zeta_0),\zeta_0)=O((d\lambda_\tau)^n).
\end{align*}
\end{proof}

\begin{proof}[Proof of Theorem \ref{thm:equ}]
The proof of Theorem \ref{thm:equ} now goes in the same way as in \cite[Section 5.9.3]{Jon15}. The only difference is that for twisted rational maps we have an additional multiplicative constant $\lambda_\tau$ in Lemma \ref{lem:laplacianpullback} and Lemma \ref{lem:distanceestimate}. It suffices to replace the corresponding formulas in \cite[Section 5.9.3]{Jon15} with these two lemmas.
\end{proof}

The following is an immediate consequence of Theorem \ref{thm:equ} and the definition of Berkovich Julia set. The second assertion can be obtained from the same argument as in \cite[Proposition 5.14]{Jon15}. 
\begin{cor}
Let $\phi_\tau$ be as in Theorem \ref{thm:equ}. Then the following holds:
\begin{enumerate}
\item For any point $\xi\in\sP^1\setminus E_{\phi_\tau}$,
\[\frac{1}{d^n}\sum_{\phi_\tau^n(\xi)=\zeta}(\deg_\xi{\phi_\tau^n})\delta_\xi\rightarrow \mu_{\phi_\tau}, \quad \text{as}\ n\rightarrow \infty.\]
\item The support of $\mu_{\phi_\tau}$ is the Berkovich Julia set $\J(\phi_\tau)$.
\end{enumerate}
\end{cor}


\section{Fatou components}
The Fatou set $\F(\phi_\tau)$ of a twisted rational map $\phi_\tau$ of degree at least $2$ is an open subset of $\sP^1$ (see Proposition \ref{prop:fatoujuliabasics}). Each connected component of $\F(\phi_\tau)$ is a \emph{(Berkovich) Fatou component} of $\phi_\tau$. We say that a component $U\subset\F(\phi_\tau)$ is \emph{periodic} if there exists $m\ge 1$ such that $\phi_\tau^m(U)=U$, and that a component $U\subset\F(\phi_\tau)$ is \emph{wandering} if $U$ has infinite forward orbit.

\subsection{Classification of periodic Fatou components}\label{sec:periodic}

We study periodic Fatou components in this section. We mainly focus on the case where $\lambda_\tau\ge1$ and $\phi_\tau$ has no type III repelling fixed points. Recall from Proposition \ref{prop:periodic-III} that if $\lambda_\tau=1$, then $\phi_\tau$ has no type III repelling fixed points.

\begin{prop}\label{prop:fix-pt}
Let $\phi_\tau$ be a twisted rational map of degree at least $2$.  Suppose that $\lambda_\tau\ge1$ and $\phi_\tau$ has no type III repelling fixed points. If $U\subset\F(\phi_\tau)$ is a fixed Fatou component of $\phi_\tau$, then $U$ contains a type I attracting fixed point or a type II indifferent fixed point. If in addition $\lambda_\tau>1$, then only the former case occurs.
\end{prop}
The proof of Proposition \ref{prop:fix-pt} follows the strategy of \cite[Lemma A.7]{DeMarco16}. A difference is that a twisted rational map may have infinitely many fixed points even in $\bP^1$. To remedy this, we use the sequential compactness of Berkovich space, see \cite[Corollaire 5.9]{Poineau13} and \cite[Corollary A]{Favre15}. 
\begin{proof}
Let $U$ be a fixed Fatou component. Assume that $U$ contains no type I attracting fixed point. We need to show that $U$ contains a type II fixed point. A type II fixed point, if exists, is necessarily indifferent by Propositions \ref{lem:non-rigid-per} (2)(3) and \ref{prop:periodic-J} (2). Note that $\phi_\tau$ has at least one fixed point $\zeta_0\in\overline{U}$ since any continuous map on a compact tree has a fixed point (see \cite{Wallace41}). If $\zeta_0$ is of type III or IV, then it is indifferent by Proposition \ref{prop:periodic-J} (3)(4) and by the assumption that $\phi_\tau$ has no type III repelling fixed points; in this case there is a type II fixed point in $U$ near $\zeta_0$, and we are done. If $\zeta_0\in\overline{U}\cap\bP^1$ is indifferent, then considering a small Berkovich disk containing $\zeta_0$ with type II boundary, we conclude that this boundary point is fixed. If $\zeta_0\in\overline{U}\cap\bP^1$ is attracting, then $\zeta_0\in U$ by Proposition \ref{prop:periodic-J} (1), which contradicts the assumption that $U$ contains no type I attracting fixed point. 

Now we work on the remaining cases where $\zeta_0\in\partial\overline{U}$ is of type II or $\zeta_0\in\overline{U}\cap\bP^1$ is repelling. 
We claim that the closure $\overline{U}$ contains a type II fixed point of $\phi_\tau$.
 If $\zeta_0\in\partial\overline{U}$ is of type II, the claim immediately holds. We focus on the case where $\zeta_0\in\overline{U}\cap\bP^1$ is repelling. Pick a small Berkovich open disk $\aD$ containing $\zeta_0$ with boundary point $\xi_0$ such that $\aD\subsetneq\phi_\tau(\aD)$. We can define a continuous map $F:\overline{U}\setminus \aD\to\overline{U}\setminus \aD$, sending $\zeta$ to $\phi_\tau(\zeta)$ if $\phi_\tau(\zeta)\notin\phi_\tau(\aD)$, and to $\phi_\tau(\xi_0)$ otherwise. It follows from \cite{Wallace41} that $F$ has a fixed point not in $\aD$, and that so does $\phi_\tau$. Repeating this process and applying the previous arguments, we conclude that $\overline{U}$ contains either a type II fixed point or a sequence of distinct repelling fixed points in $\overline{U}\cap\bP^1$. In the latter case, by \cite[Corollaire 5.9]{Poineau13}, passing to a subsequence if necessary, we obtain a limit fixed point $\zeta_\infty$ of $\phi_\tau$ in $\overline{U}$. If $\zeta_\infty$ is not of type I, then we conclude by previous arguments. Therefore we assume that $\zeta_\infty$ is of type I. If $\zeta_\infty$ is in the boundary of $U$ then it is in the Julia set $\J(\phi_\tau)$. By Proposition \ref{prop:periodic-J} (1) and Theorem \ref{thm:F-J} $\zeta_\infty$, a type I fixed point in $\J(\phi_\tau)$ is repelling, which contradicts the fact that $\zeta_\infty$ is a limit of other repelling fixed points. Therefore $\zeta_\infty$ is necessarily contained in $U$ and hence is indifferent by Proposition \ref{prop:periodic-J} (1) and by the assumption that $U$ contains no type I attracting fixed point. Then we can find a type II fixed point in $U$ as in the previous arguments.  

We can now assume that $\zeta_0\in\overline{U}$ is a type II fixed point. It is either indifferent or repelling by Proposition \ref{lem:non-rigid-per} (2) and (3). If $\zeta_0\in U$, we immediately obtain the desired conclusion. Now let us consider the case where $\zeta_0\in\partial U\cap\sH^1$. Since $U$ is fixed, the direction $\vec{v}$ at $\zeta_0$ containing $U$ is fixed by $T_{\zeta_0}\phi_\tau$. If $\zeta_0$ is indifferent, then any type II point in the direction $\vec{v}$ sufficiently close to $\zeta_0$ is a fixed point and we are done. Assume now that $\zeta_0$ is repelling. We argue according to the number of repelling fixed points in the boundary $\partial U$. If $\partial U$ contains finitely many repelling fixed points, then removing finitely many Berkovich open disks containing these fixed points and applying a similar argument as in the previous paragraph, we conclude that $U$ contains a type II fixed point.  If $\partial U$ contains infinitely many repelling fixed points, again by  \cite[Corollaire 5.9]{Poineau13}, passing to a subsequence if necessary, we obtain a limit fixed point $\xi_\infty\in\partial U\subset\J(\phi_\tau)$. By Proposition \ref{prop:periodic-J} and by the assumption that $\phi_\tau$ has no type III repelling fixed points, the fixed point $\xi_\infty$ is of type I or II. Since all other boundary points of $U$ are contained in a same direction at $\xi_\infty$, we assert that the direction $\vec{w}$ at $\xi_\infty$ containing $U$ is a not a repelling direction because otherwise there would not exist sequences of distinct fixed points converging to $\xi_\infty$. Since $\lambda_\tau\ge 1$, the nonrepelling direction $\vec{w}$ is necessarily indifferent and $\lambda_\tau$ is in fact $1$. It follows that there exists a type II fixed point sufficiently close to $\xi_\infty$ in the direction $\vec{w}$. This completes the proof of the first assertion in the proposition. 

If $\lambda_\tau>1$, then any type II fixed point in $U$ is repelling. It follows that $U$ can not contain type II fixed point by Proposition \ref{prop:periodic-J} (2).  Thus the second assertion holds.
\end{proof}

Now we focus on tame twisted rational map $\phi_\tau$ with $\lambda_\tau=1$ and show that any fixed component in $\F(\phi_\tau)$  containing a type II fixed point has a periodic skeleton.

\begin{prop}\label{prop:periodic-sk}
Let $\phi_\tau$ be a tame twisted rational map of degree at least $2$. Assume that $\lambda_\tau=1$ and that $U\subset\F(\phi_\tau)$ is a fixed Fatou component containing a type II fixed point $\zeta_0\in U$. Then for any $\xi_0\in\partial U$, each point of the segment $[\xi_0, \zeta_0]$ is periodic, i.e.\ there exists $m\ge 1$ such that $\phi_\tau^m(\zeta)=\zeta$ for all $\zeta\in[\xi_0,\zeta_0]$.
\end{prop}

The proof of Proposition \ref{prop:periodic-sk} is essentially the same as \cite[Lemma A.7]{DeMarco16}. It suffices to replace, in the proof of \cite[Lemma A.7]{DeMarco16}, the parts corresponding to the next two lemmas with them. We omit here the detailed proof. 

\begin{lem}\label{lem:tangent1}
 Let $\psi\in K(z)$ be represented by a convergent power series and let $\psi_\tau$ be a twisted rational map of degree $\geq 1$. Suppose that $\psi_\tau:\aD(0,1)\to \aD(0,1)$ is bijective. Then for any classical point $x_0$ and $y\not=0$ in $\aD(0,1)\cap K$, for the direction $\vec{v}\in T_{\xi_G}\sP^1$ containing $0$ and for any $z\in\aD(0,1)\cap K$, we have 
$$\psi_\tau(x_0+y(z+\vec{v}))=\psi_\tau(x_0)+\psi'(0)\tau(yz)+\tau(y)\vec{v},$$
where $a+b\vec{v}$ means $a+b(\sB(\vec{v})\cap\bP^1)$ for any $a,b\in K$.
\end{lem}

\begin{proof}
Write $\psi(z)=\sum_{n\ge 0}a_nz^n$. Then  $\psi_\tau(z)=\sum_{n\ge 0}a_n(\tau(z))^n$. Since both $\psi_\tau$ and $\tau$ are bijective from $\aD(0,1)$ to itself, so is $\psi$. It follows that $|a_0|<1$, $|a_1|=1$ and $|a_j|\le1$ for all $j\ge 2$. 
We compute  
\begin{align*}
\psi_\tau(x_0+y(z+\vec{v}))&-\psi_\tau(x_0)=a_1\tau(y(z+\vec{v}))+\sum_{n\ge 2}a_n\left((\tau(x_0+y(z+\vec{v})))^n-(\tau(x_0))^n\right)\\
&=a_1\tau(yz)+a_1\tau(y\vec{v})+\tau(y\vec{v})=a_1\tau(yz)+\tau(y\vec{v})=\psi'(0)\tau(yz)+\tau(y)\vec{v}.
\end{align*}
The conclusion follows. 
\end{proof}

 \begin{lem}\label{lem:replace}
Let $\phi_\tau$ be a twisted rational map of degree at least $2$. Assume that $\zeta\in\sP^1$ is a type II fixed point of $\phi_\tau$ and that $\vec{v}\in T_{\zeta}\sP^1$ is a good direction at $\zeta$ fixed by $T_\zeta\phi_\tau$. If there exists a Berkovich closed $\overline{\aD}\subsetneq\sB(\vec{v})$ such that $\phi_\tau:\sB(\vec{v})\setminus\overline{\aD}\to\sB(\vec{v})$ is injective, then there exists an injective power series $h:\sB(\vec{v})\to\sB(\vec{v})$ such that $h\circ\tau(\zeta)=\phi_\tau(\zeta)$ for all $\zeta\in\sB(\vec{v})$ with $\diam(\zeta)\ge\diam(\overline{\aD})$.
 \end{lem}
 \begin{proof}
 The conclusion follows from Rivera's approximation lemma \cite[Section 5]{Rivera03II} (see also \cite[Lemma A.4]{DeMarco16}), applied to the composition of $\phi: \tau(\sB(\vec{v}))\to\sB(\vec{v})$ and an $M\in\PGL(2,K)$ mapping $\sB(\vec{v})$ to $\tau(\sB(\vec{v}))$.
 \end{proof}

Inspired by Propositions \ref{prop:fix-pt} and \ref{prop:periodic-sk}, for a fixed Fatou component $U\subset\sP^1$, we say that $U$ is an \emph{attracting domain} if $U$ contains a classical attracting fixed point, and that $U$ is a \emph{Rivera domain} if $\phi_\tau:U\to U$ is bijective.

\begin{proof}[Proof of Theorem \ref{thm:periodic}]
With the preparatory propositions in this section, Kiwi's proof of \cite[Lemma A.1]{DeMarco16} can be directly transported to the case of twisted rational maps satisfying the hypothesis of Theorem \ref{thm:periodic}. 
\end{proof}

\subsection{Wandering domains}\label{sec:wandering}

To prove Theorem \ref{thm:nonwandering}, we begin with the the following two lemmas, which are similar to the case of rational maps.

 \begin{lem}\label{lem:distance}
 Let $\phi_\tau$ be a twisted rational map with $\deg\phi_\tau\ge 1$ and $\tau\in\Aut^\ast(K)$. Consider a connected Berkovich affinoid  $U\subset\sP^1$ with at least two boundary points. Let $\delta>0$ be the minimum hyperbolic distance between two distinct boundary points of $U$. Then $\rho(\zeta,\xi)\ge\lambda_\tau\delta$ for any two distinct points $\zeta,\xi\in\partial\phi_\tau(U)$.
\end{lem}
\begin{proof}
By Lemma \ref{lem:fund} (2), we have  $\rho(\zeta_1,\xi_1)\ge\lambda_\tau\delta$ for any distinct $\zeta_1,\xi_1\in\tau(U)$. 
Then the conclusion follows from  \cite[Proposition 11.3]{Ben19}.
\end{proof}

\begin{lem}\label{prop:disk}
Let $\phi_\tau$ be a twisted rational map with $\deg\phi_\tau\ge 2$.  Assume  $U\subset\F(\phi_\tau)$ is a wandering domain. Then there exists $N\ge 0$ such that $\phi_\tau^n(U)$ is a Berkovich open disk for all $n\ge N$.
\end{lem}
\begin{proof}
The proof of \cite[Theorem 11.2]{Ben19} can be applied here by Lemma \ref{lem:distance}.
\end{proof}

\begin{proof}[Proof of Theorem \ref{thm:nonwandering}]
Let $U\subset\sP^1$ be a wandering domain of $\phi_\tau$ containing a $\overline{k}$-rational point $b$. Extending $k$ to $k(b)$ and conjugating by an element $\PGL(2,k(b))$, we may assume that $b\in k$ and that the wandering domain $U$ of $\phi_\tau$ contains $\sP^1\setminus\overline{\aD}(0,1)$. Up to replace $k$ with a finite extension, we can assume that $k$ contains all the poles and critical points of $\phi_\tau$. 
By the existence of $L$, we can further assume that $\tau(k)\subseteq k$. 
Applying the argument of \cite[Theorem 11.23]{Ben19} for $\phi_\tau$, we conclude that $\partial U=\{\zeta\}$ for some type II (pre)periodic point $\zeta\in\sH^1$ and hence $U$ is contained in the basin of type II Julia cycle that is contained in the forward orbit of $\zeta$.
\end{proof}

\bibliographystyle{siam}
\bibliography{references}

\end{document}